\documentclass[a4paper]{amsart}

\pdfoutput=1

\usepackage{stackrel}
\usepackage{amssymb}
\usepackage{amsmath}
\usepackage{stmaryrd}
\usepackage{mathtools}
\usepackage{mathrsfs}
\usepackage{microtype}
\usepackage{lmodern}
\usepackage[T1]{fontenc}
\usepackage[all]{xy}
\usepackage{tikz}
\usepackage{bm}

\newtheorem{theorem}{Theorem}[section]
\newtheorem*{mt}{Main Theorem}

\newtheorem{lemma}[theorem]{Lemma}
\newtheorem{prop}[theorem]{Proposition}

\newtheorem{cor}[theorem]{Corollary}

\theoremstyle{definition}
\newtheorem{remark}[theorem]{Remark}
\newtheorem{example}[theorem]{Example}

\numberwithin{equation}{section}

\DeclareMathOperator*{\GL}{GL}
\DeclareMathOperator*{\SL}{SL}
\DeclareMathOperator*{\Cl}{Cl}
\DeclareMathOperator*{\Pic}{Pic}
\DeclareMathOperator*{\Div}{Div}
\DeclareMathOperator*{\Sym}{Sym}
\DeclareMathOperator*{\cone}{cone}
\DeclareMathOperator*{\fan}{fan}
\DeclareMathOperator*{\Spec}{Spec}
\DeclareMathOperator*{\Hom}{Hom}
\DeclareMathOperator*{\trop}{trop}

\DeclareMathOperator*{\Supp}{Supp}
\DeclareMathOperator*{\relint}{relint}

\newcommand{\ord}{{\rm ord}}

\newcommand{\rleft}{\mathopen{}\mathclose\bgroup\left}
\newcommand{\rright}{\aftergroup\egroup\right}

\newcommand{\A}{\mathbb{A}}
\newcommand{\C}{\mathbb{C}}
\newcommand{\Q}{\mathbb{Q}}
\newcommand{\N}{\mathbb{N}}
\newcommand{\V}{\mathbb{V}}
\newcommand{\T}{\mathbb{T}}
\newcommand{\Z}{\mathbb{Z}}

\newcommand{\I}{\mathbb{I}}
\newcommand{\X}{\mathfrak{X}}

\newcommand{\Pb}{\mathbb{P}}

\newcommand{\Om}{\mathcal{O}}

\newcommand{\Vm}{\mathcal{V}}

\newcommand{\Fm}{\mathcal{F}}
\newcommand{\Dm}{\mathcal{D}}

\newcommand{\Tm}{\mathcal{T}}
\newcommand{\Rm}{\mathcal{R}}

\newcommand{\Nm}{\mathcal{N}}
\newcommand{\Mm}{\mathcal{M}}

\newcommand{\Nf}{\mathfrak{N}}
\newcommand{\Lf}{\mathfrak{L}}

\newcommand{\orho}{\rho}

\newcommand{\p}{\mathfrak{p}}

\newcommand{\hi}{V}
\newcommand{\un}{T}

\newcommand{\glob}[1]{\Gamma( #1, \Om_{#1} )}

\newcommand{\arb}[1]{\bm{ #1 }}

\begin{document}

\title{The Cox ring of a spherical embedding}

\author{Giuliano Gagliardi}
\address{Mathematisches Institut, Universit\"at T\"ubingen, Auf der
Morgenstelle 10, 72076 T\"ubingen, Germany}
\curraddr{}
\email{giuliano.gagliardi@uni-tuebingen.de}
\thanks{}

\date{}

\dedicatory{}

\begin{abstract}
Let $G$ be a connected reductive group and $G/H$ a 
spherical homogeneous space.
We show that the ideal of relations between a natural set of
generators of the Cox ring of a $G$-embedding of $G/H$
can be obtained by homogenizing certain
equations which depend only on the homogeneous space.
Using this result, we describe some examples of spherical homogeneous
spaces such that the Cox ring of any of their $G$-embeddings
is defined by one equation.
\end{abstract}

\maketitle

\section*{Introduction}
Throughout the paper, we work with algebraic varieties and
algebraic groups over the field of complex numbers $\C$.

Let $Y$ be a normal irreducible variety whose divisor
class group $\Cl(Y)$ is finitely generated and
$\Gamma(Y, \Om_Y^*) = \C^*$. The \em Cox ring \em of $Y$
is the $\Cl(Y)$-graded $\C$-algebra 
\begin{align*}
\Rm(Y) \coloneqq \bigoplus_{[D] \in \Cl(Y)} \Gamma(Y, \Om_Y(D))\text{,}
\end{align*}
with the multiplication defined by the canonical maps
\begin{align*}
\Gamma(Y, \Om_Y(D_1)) \otimes \Gamma(Y, \Om_Y(D_2)) \to \Gamma(Y, \Om_Y(D_1+D_2))\text{.}
\end{align*}
Some accuracy is required in order for this
multiplication to be well-defined
(cf.~\cite{ha} or \cite{coxrings} for details).

It has been shown by Cox that $\Rm(Y)$ is a polynomial
ring if $Y$ is a toric variety (cf.~\cite{tcox}). The
converse was obtained by Hu and Keel for smooth
projective varieties (cf.~\cite[Corollary~2.10]{huke}).
Toric varieties can be considered
as examples in the more general class of spherical
varieties, which are quasihomogeneous with respect
to the action of a connected reductive group $G$.

The aim of this paper is to investigate the Cox ring of an arbitrary
spherical variety.

We recall some standard notions from the theory of
spherical varieties. 
Let $G$ be a connected reductive group.
A closed subgroup $H \subseteq G$ is called
\em spherical \em if a Borel subgroup $B \subseteq G$
has an open orbit in $G/H$. Then $G/H$
is called a \em spherical homogeneous space\em.
In this case, we may always assume that $BH$ is open in $G$.
A $G$-equivariant open
embedding $G/H \hookrightarrow Y$ into a normal irreducible $G$-variety $Y$
is called a \em spherical embedding\em, and $Y$ is called a \em spherical variety\em.
Let $p: G' \to G$ be a finite covering
such that $G'$ is of simply connected 
type, i.e.~$G' = G^{ss} \times C$ where $G^{ss}$ is semisimple simply
connected and $C$ is a torus. Then $G/H \cong G'/p^{-1}(H)$,
and these two spherical homogeneous spaces have exactly the same embeddings.
Therefore in this paper we will always assume $G = G^{ss} \times C$.
Similarly to the theory of toric varieties, spherical embeddings
$G/H \hookrightarrow Y$ can be described by some combinatorial
data introduced by Luna and Vust (cf.~\cite{lunavust} and \cite{knopsph}).

We denote by $\Mm$ the weight lattice of $B$-eigenvectors in the function
field $\C(G/H)$ and by $\Nm \coloneqq \Hom(\Mm, \Z)$ the dual lattice
with the natural pairing $\langle \cdot , \cdot \rangle: \Nm \times \Mm \to \Z$.
If $\nu: \C(G/H) \to \Q$ is a discrete valuation, its
restriction to $B$-eigenvectors induces a map $u: \Mm \to \Q$,
which lies in the vector space $\Nm_\Q \coloneqq \Nm \otimes_\Z \Q$.
We denote the set of $G$-invariant discrete valuations on $\C(G/H)$ by $\Vm$.
The above assignment is injective on $\Vm$ and therefore defines an inclusion
$\Vm \subseteq \Nm_\Q$. The set $\Vm$ is a polyhedral convex cone called the \em valuation
cone \em of $G/H$. In fact, the cone $\Vm$ is the intersection of half-spaces
$\{u \in \Nm_\Q : \langle u, \gamma_i \rangle \le 0\}$ where
$\gamma_1, \ldots, \gamma_s$ are linearly independent primitive elements in the
lattice $\Mm$ called the \em spherical roots \em
of $G/H$.

A spherical embedding $G/H \hookrightarrow Y$ is called a \em wonderful
completion \em if $Y$ is complete, smooth, and contains exactly one closed $G$-orbit.
Then $Y$ is called a \em wonderful variety\em.
A wonderful completion of $G/H$ exists (and is unique) if and only if the valuation cone $\Vm$ is spanned by a
basis of $\Nm$. 
To any spherical homogeneous space $G/H$ one can associate in a natural way
a closed subgroup $H' \times C \subseteq N_G(H)$ of finite index and containing $H$
such that a wonderful completion $G/(H' \times C) \cong G^{ss}/H' \hookrightarrow Y'$ exists (cf.~\cite{Luna:typea}).

Let $G/H \hookrightarrow Y$
be a spherical embedding.
As the Cox ring does not depend on $G$-orbits of codimension
two or greater, we may assume that $Y$ contains only non-open
$G$-orbits of codimension one, which are exactly the
$G$-invariant prime divisors $Y_1, \ldots, Y_n$ in $Y$. To each $Y_l$
we assign the primitive element $u_l \in \Nm$ corresponding
to the discrete valuation on $\C(G/H)$ induced by $Y_l$.
It follows from the Luna-Vust theory that the embedding $G/H \hookrightarrow Y$
can be described combinatorially by the fan $\Sigma$ in $\Nm_\Q$ 
consisting exactly of the trivial cone $0$ and the one-dimensional cones $\sigma_l \coloneqq \Q_{\ge 0}u_l$,
which lie in the valuation cone $\Vm$.

The Cox ring of $Y$
has been described by Brion (cf.~\cite{brcox}).
In fact, he describes the equivariant Cox ring $\Rm_G(Y) \cong \Rm(Y) \otimes_\C \C[C]$.
As a first step, he computes the Cox ring $\Rm(Y')$ of the associated wonderful completion $G^{ss}/H' \hookrightarrow Y'$
as the Rees algebra of a certain filtration of the ring $\glob{G^{ss}/K'}$ where $K'$ is
the intersection of the kernels of all characters of $H'$.
The equivariant Cox ring $\Rm_G(Y)$ is then obtained from $\Rm(Y')$ by a base change.
There is a natural bijection between the set $\{\gamma_1, \ldots, \gamma_s\}$
of spherical roots of $G^{ss}/H'$ and the set $\{Y'_1, \ldots, Y'_s\}$ of
$G$-invariant prime divisors in $Y'$.
It immediately follows from the result of Brion that
\begin{align*}
\Rm(Y) \cong \Rm(Y') \otimes_{\C[Z_1, \ldots, Z_s]} \C[W_1, \ldots, W_n]
\end{align*}
where $\C[Z_1, \ldots, Z_s]$ and $\C[W_1, \ldots, W_n]$ are polynomial rings.
The homomorphism $\C[Z_1, \ldots, Z_s] \to \Rm(Y')$ sends
$Z_i$ to the canonical section in $\Gamma(Y', \Om_{Y'}(Y'_i))$,
and the homomorphism $\C[Z_1, \ldots, Z_s] \to \C[W_1, \ldots, W_n]$
sends
\begin{align*}
Z_i \mapsto \prod_{l=1}^n W_l^{-\langle u_l, \gamma_i \rangle}\text{.}
\end{align*}

Our main result is another description of the Cox ring $\Rm(Y)$.
We will reduce the general case
to the case where the spherical homogeneous space $G/H$
has trivial divisor class group and is therefore quasiaffine.
We show that the Cox ring $\Rm(Y)$ is obtained by homogenizing the
equations of the affine closure of $G/H$ inside an appropriately constructed 
embedding $G/H \hookrightarrow \C^d \times (\C^*)^m$.
As a byproduct,
we obtain a description of the valuation cone $\Vm$
in terms of tropical algebraic geometry.
We will compare our results with the approach
of Brion in Section~\ref{five}.

We now give a detailed summary of our results. 
We assume that the homogeneous space $G/H$ has
trivial divisor class group.
Let $\Dm \coloneqq \{D_1, \ldots, D_r\}$ be the set of $B$-invariant
prime divisors in $G/H$, and assume that the spherical embedding $G/H \hookrightarrow Y$
satisfies $\Gamma(Y, \Om_Y^*) = \C^*$.
We choose prime elements $f_1, \ldots, f_r \in \Gamma(G/H, \Om_{G/H})$
with $\V(f_i) = D_i$ and obtain irreducible $G$-modules $V_i \coloneqq \langle G \cdot f_i
\rangle \subseteq \Gamma(G/H, \Om_{G/H})$.
For each $i$ we set $s_i \coloneqq \dim V_i$ and choose a basis $\{f_{ij}\}_{j=1}^{s_i} \subseteq
G \cdot f_i$ of $V_i$ with $f_{i1} = f_i$.
We let $m$ be the rank of the finitely generated free abelian group
$\Gamma(G/H, \Om^*_{G/H})/\C^* \cong \X(G)^H \cong \X(C)^{p_2(H)}$ where $p_2: G \to C$ denotes the projection
and choose representatives $\{g_k\}_{k=1}^m \subseteq \Gamma(G/H, \Om^*_{G/H})$ of a basis.
The $B$-weights of $f_1, \ldots, f_r$ and $g_1, \ldots, g_m$ freely generate lattices $\Mm_\hi$ and $\Mm_\un$ respectively, and
we obtain $\Mm = \Mm_\hi \oplus \Mm_\un \cong \Z^{r+m}$.
Finally, we define the torus $T \coloneqq C/p_2(H) \cong \Spec(\C[\Mm_\un])$.

The next step is to define a $G$-equivariant locally closed embedding
\begin{align*}
G/H \hookrightarrow Z \coloneqq V^* \times T \cong \C^{s_1+\ldots+s_r} \times (\C^*)^m
\end{align*}
where $V^*$ is dual to $V \coloneqq V_1 \oplus \ldots \oplus V_r$.
We have $\C[Z] = S(V) \otimes_\C \C[\Mm_\un]$, where
$S(V)$ denotes the symmetric algebra of $V$.
The coordinate ring of $V_i^*$ is the symmetric algebra $S(V_i)$, whose
generators corresponding to the above basis we denote by $S_{ij}$ for $1 \le j \le s_i$,
i.e.~$\C[V_i^*] = S(V_i) = \C[S_{i1}, \ldots, S_{is_i}]$.
We denote the generators of the coordinate ring of $T$ corresponding to the
above basis of $\Mm_\un$ by $T_k$ for $1 \le k \le m$, i.e.~$\C[T] = \C[\Mm_\un] = \C[T_1^{\pm 1}, \ldots,
T_m^{\pm 1}]$.
The locally closed embedding $G/H \hookrightarrow Z$ is then given
by the $G$-equivariant surjective map $\C[Z] \to \Gamma(G/H, \Om_{G/H})$
sending
$S_{ij} \mapsto f_{ij}$ and
$T_k \mapsto g_k$.
Its kernel is the prime ideal $\I({G/H})$.
Considering the natural action of the torus $\Spec(\C[\Mm])$ on $Z$, we obtain
a corresponding $\Mm$-grading on the coordinate ring $\C[Z]$. For $f \in \C[Z]$ and
$\mu \in \Mm$ we denote the $\mu$-homogeneous
component of $f$ by $f^{(\mu)}$.

In order to describe the relations of $\Rm(Y)$, we define a homogenization operation
in two steps. The first step is the map $\alpha: \C[Z] \to \big(\C[Z]\big)[W_1, \ldots, W_n]$
defined as follows.
For each $f \in \C[Z]$ and $u \in \Nm$ we define
\begin{align*}
\ord_u(f) \coloneqq \min_{\mu\in\Mm} \rleft\{\rleft\langle u, \mu \rright\rangle; f^{(\mu)} \ne 0\rright\}\text{,}
\end{align*}
and set
\begin{align*}
f^\alpha \coloneqq \frac{\sum_{\mu \in \Mm} \rleft(f^{(\mu)} \prod_{l=1}^n W_l^{\rleft\langle u_l, \mu \rright\rangle}\rright)}
{\prod_{l=1}^n W_l^{\ord_{u_l}(f)}}\text{.}
\end{align*}
The second step is the map $\beta: \big(\C[Z]\big)[W_1, \ldots, W_n]
\to S(V)[W_1, \ldots, W_n]$ sending
$T_k \mapsto 1$ for each $1 \le k \le m$.
Finally, we define the map $h: \C[Z] \to S(V)[W_1, \ldots, W_n]$
by composing the two steps, i.e.~$h \coloneqq \beta \circ \alpha$.
Note that we write the application of the maps $\alpha$, $\beta$, and $h$
as exponents, for example we write $f^h$ instead of $h(f)$.
We can now state our main result.

\begin{mt}
We have
\begin{align*}
\Rm(Y) \cong S(V)[W_1, \ldots, W_n]\rleft/\middle(f^h; f \in \I(G/H)\rright)\text{,}
\end{align*}
with $\Cl(Y)$-grading given by $\deg(S_{ij}) = [D_i]$ and $\deg(W_l) = [Y_l]$.
If $H$ is connected, $\Rm(Y)$ is a factorial ring.
\end{mt}

In particular, if $\I(G/H) = (f)$ is a principal ideal generated
by $f$, then the ideal of relations of $\Rm(Y)$ is generated by $f^h$.

In the special case of a toric variety $Y$, the Main Theorem reduces to the result of
Cox that $\Rm(Y)$ is a polynomial ring with one variable per
$G$-invariant prime divisor because the homogeneous space
has no $B$-invariant prime divisors,
i.e.~$V = \langle 0\rangle$, and $\I(G/H) = (0)$.
In the special case of a horospherical variety $Y$, we will see that
$\Rm(Y) \cong \Rm(G/P)[W_1, \ldots, W_n]$,
where $P \coloneqq N_G(H)$,
i.e.~the Cox ring of $Y$ is a polynomial
ring over the Cox ring of $G/P$. This also follows
directly from the description of the Cox ring by Brion.

The paper is organized in five sections. 
In Section~\ref{one} we present some 
information about the valuation cone $\Vm$ and
describe it using tropical algebraic geometry.
The proof of the Main Theorem in the crucial case of a
spherical homogeneous space with trivial divisor class group 
is given in Section~\ref{two}.
Then we explain in Section~\ref{three} how the results can be extended
to arbitrary spherical homogeneous spaces.
In Section~\ref{four}, we illustrate our results by some explicit examples.
Finally, we compare our results with
the approach of Brion in Section~\ref{five}.

\section{The valuation cone $\Vm$}
\label{one}

We continue to use the notation and the assumptions from the introduction in
this section.

It was shown in \cite{brionsym} and \cite{knopsym} that
the valuation cone $\Vm$ is a fundamental chamber for the
action of a crystallographic reflection group $W_{G/H}$
on $\Nm_\Q$ called the \em little Weyl group \em (cf.~\cite[Theorem~22.13]{ti}).
In particular, the cone $\Vm$ is cosimplicial,
and the set of spherical roots is the
set of simple roots of a root system
with Weyl group $W_{G/H}$.
We have the isotypic decomposition into $G$-modules
\begin{align*}
\glob{G/H} = \bigoplus_{\mu} V_{\mu}\text{,}
\end{align*}
where $\mu$ runs over pairwise distinct elements of $\Mm$
and $V_{\mu}$ is an irreducible $G$-module of highest weight $\mu$.
Let $\Tm$ be the cone in $\Mm \otimes_\Z \Q$
generated by the elements of the form $\mu_1 + \mu_2 - \mu_3$ such that
\begin{align*}
V_{\mu_3} \subseteq V_{\mu_1}\cdot V_{\mu_2}\text{.}
\end{align*}
It follows from \cite[Proposition~2.13]{vbmod} that the cone $\Tm$
is polyhedral. The cone $\Tm$ is equal to the cone
generated by the spherical roots,
i.e.~$-\Vm$ is the dual cone of $\Tm$ (cf.~\cite[Section~3]{losev} and \cite[Section~5]{knopsph}).

In this sense, the valuation cone $\Vm$ is related to the failure of
the $\C$-algebra $\glob{G/H}$ being $\Mm$-graded. 
We have the following extreme case: a spherical homogeneous space
$G/H$ is called \em horospherical \em if $\Vm = \Nm_\Q$. In our setting,
this is equivalent to the $\C$-algebra $\glob{G/H}$ being $\Mm$-graded.
Another characterization of horospherical homogeneous spaces is
that $H$ contains a maximal unipotent subgroup of $G$ (cf.~\cite{pauer2} and \cite[Corollary 6.2]{knopsph}).

We will state a result about the valuation cone which
uses tropical algebraic geometry. An introduction to this subject
can be found in \cite{tropintro}. Let $X$ be the toric variety
associated to a fan $\Sigma_X$ in $N_\Q$ with torus $\T$. If $S \subset \T$ is a closed subset,
the tropicalization $\trop(S) \subseteq N_\Q$ of $S$ is the support of
a polyhedral fan in $N_\Q$. It gives an answer to the question
of which $\T$-orbits in $X$ intersect the closure $\overline{S}$
of $S$ inside $X$:
by a result of Tevelev (cf.~\cite{tev}), the $\T$-orbit
corresponding to $\sigma \in \Sigma_X$ intersects $\overline{S}$ if
and only if $\trop(S)$ intersects $\relint(\sigma)$.

Our plan is to use this machinery on the embedding $G/H \hookrightarrow Z$.
First, we show that this is indeed a locally closed embedding.
Along the way, we also prove some other claims which have been stated in the introduction.

\begin{remark}
By \cite[Proposition 1.3]{knoppic}, 
every $f_i$ is a $B$-eigenvector,
all elements of $\Gamma(G/H, \Om^*_{G/H})$ are $G$-eigenvectors, and the
quotient $\Gamma(G/H, \Om^*_{G/H})/\C^*$ is a finitely generated free abelian group.
Moreover,
for each $f_i$ the $G$-module spanned by $G \cdot f_i$ in $\glob{G/H}$
is irreducible and finite-dimensional (cf.~\cite[III.1.5]{kraftinv}).
\end{remark}

\begin{prop}
\label{prop-b-eigenvectors}
The $B$-eigenvectors in $\glob{G/H}$ are given by
\begin{align*}
\glob{G/H}^{(B)} = \rleft\{cf_1^{d_1}\cdots f_r^{d_r}; c\in \Gamma(G/H, \Om^*_{G/H}), d_i \in \N_0 \rright\}\text{.}
\end{align*}
\end{prop}
\begin{proof}
For each $f \in \glob{G/H}^{(B)}$ all irreducible components of $\V(f)$ are $B$-invariant since
$B$ is irreducible. 
\end{proof}
\begin{prop}
\label{generators}
The $\C$-algebra $\glob{G/H}$ is generated by $\{f_{ij}, g_k^{\pm 1}\}$.
\end{prop}

\begin{proof}
Since every element of $\glob{G/H}^U$ can be written as the sum of $B$-eigenvectors,
the $\C$-algebra $\glob{G/H}^U$ is generated by $\{f_i,g_k^{\pm 1}\}$. By
\cite[III.3.1]{kraftinv}, it follows that $\glob{G/H}$ is generated as claimed.
\end{proof}

For each $f \in \C(G/H)^{(B)}$ we denote by $\chi_f$ its $B$-weight.
For the weight lattice $\Mm$ of $B$-eigenvectors in $\C(G/H)$ we then have
\begin{align*}
\Mm &= \{\chi \in \X(B); \text{there exists } f\in \C(G/H)^{(B)} \text{ with } \chi=\chi_f\}\text{.}
\end{align*}
Recall that there is an exact sequence (cf.~\cite[after 1.7]{knopsph})
\begin{align*}
1 \to \C^* \to \C(G/H)^{(B)} \to \Mm \to 0\text{.}
\end{align*}
We define $v_i^* \coloneqq \chi_{f_i}$ and $w_k^* \coloneqq \chi_{g_k}$.

\begin{prop}
\label{prop:m-basis}
The lattice $\Mm \subseteq \X(B)$ is freely generated by
\begin{align*}
\{v_1^*, \ldots, v_r^*, w_1^*, \ldots, w_m^*\}\text{.}
\end{align*}
\end{prop}
\begin{proof}
Every $f \in \C(G/H)^{(B)}$ can be written as $\frac{g}{h}$ with
$g, h \in \glob{G/H}^{(B)}$ since $\Supp\rleft(\operatorname{div}(f)\rright)$ is
the union of $B$-invariant prime divisors. The claim then follows from
Proposition~\ref{prop-b-eigenvectors}, the exact sequence above,
and the fact that $\{g_k\}_{k=1}^m$ is a basis of $\Gamma(G/H, \Om^*_{G/H})/\C^*$.
\end{proof}

We denote the corresponding dual basis of $\Nm$ by
$\{v_1, \ldots, v_r, w_1, \ldots, w_m\}$. We also have
$\Mm_\hi = \langle v_1^*, \ldots, v_r^* \rangle$ and
$\Mm_\un = \langle w_1^*, \ldots, w_m^* \rangle$.

\begin{lemma}
\label{le:alg}
Let $0 \ne h \in \Gamma(G/H, \Om_{G/H})$ and $g \in G$
such that $g\cdot h = ch$ where $c$ is a unit. Then we have $c \in \C$.
\end{lemma}
\begin{proof}
We recursively define $h_0 \coloneqq h$ and $h_{i+1} \coloneqq g \cdot h_i$
for $i \in \N_0$. As $h$ is contained in a finite-dimensional $G$-module,
there exists $d \in \N$ such that $\lambda_0h_0 + \ldots +\lambda_d h_d=0$
for some coefficients $\lambda_i \in \C$. As $c$ is a $G$-eigenvector
(cf.~\cite[Proposition 1.3]{knoppic}), dividing by $h$ yields
a polynomial relation for $c$. Hence $c$ is algebraic over $\C$.
\end{proof}

\begin{prop}
The $f_{ij}$ are pairwise nonassociated prime elements,
\end{prop}
\begin{proof}
For every $f_{ij}$ we have $f_{ij} \in G \cdot f_i$ and $f_i$ is prime, hence $f_{ij}$ is also prime.

Now let $f_{i_1j_1} = cf_{i_2j_2}$ for some unit $c$. Then $g\cdot f_{i_1} = c'f_{i_2}$ for some unit $c'$ and $g \in G$.
By Proposition~\ref{prop:m-basis}, we have $i_1 = i_2$, so there exists $g' \in G$
with $g'\cdot f_{i_2j_2} = f_{i_1j_1} = cf_{i_2j_2}$. By Lemma~\ref{le:alg}, it follows that $c \in \C$, therefore $j_1 = j_2$.
\end{proof}

The map defined in the introduction
\begin{align*}
\Phi: S(V) \otimes_\C \C[T] &\to \glob{G/H} \\
S_{ij} &\mapsto f_{ij}\\
T_k &\mapsto g_k\text{,}
\end{align*}
is surjective by Proposition~\ref{generators} and $G$-equivariant, 
hence indeed induces a locally closed $G$-equivariant
embedding $G/H \hookrightarrow Z$
with respect to the corresponding $G$-action on $Z$.

Finally, we have to explain how $Z$ can be naturally regarded as a toric variety.
We denote by $M$ the finitely generated free abelian group with
basis 
\begin{align*}
\{S_{ij}, T_k; 1 \le i \le r, 1 \le j \le s_i, 1 \le k \le m\}\text{,}
\end{align*}
which is isomorphic to $\Z^{s_1+\ldots+s_r+m}$. We define the torus $\T \coloneqq \Spec(\C[M])$ with character lattice $M$, denote the dual lattice by $N \coloneqq \Hom(M, \Z)$,
denote the corresponding dual basis by
\begin{align*}
\{v_{ij}, w_k; 1 \le i \le r, 1 \le j \le s_i, 1 \le k \le m\}\text{,}
\end{align*}
 and set $N_\Q \coloneqq N \otimes_\Z \Q$.
The surjective map $M \to \Mm$ sending $S_{ij} \mapsto v_i^*$ and 
$T_k \mapsto w_k^*$ induces an inclusion $\Nm_\Q \hookrightarrow N_\Q$
sending $v_i \mapsto v_{i1} + \ldots + v_{is_i}$ and $w_k \mapsto w_k$.
The action of $\T$ on $Z$ makes $Z$ a toric variety.

\begin{theorem}
\label{prop-trop}
We have
\begin{align*}
\Vm = \trop(G/H \cap \T) \cap \Nm_\Q \text{.}
\end{align*}
\end{theorem}

The proof will be given at the end of Section~\ref{two}.
We can compare this result to the approach Luna and Vust introduced in
\cite{lunavust} using formal curves (cf.~\cite[Chapter~24]{ti}),
which shows that every $G$-invariant discrete valuation $\nu$
can be obtained up to proportionality by choosing a $\C(\!(t)\!)$-valued point $x(t)$ of $G/H$
and defining
\begin{align*}
\nu(f) \coloneqq \ord(f(g\cdot x(t)))
\end{align*}
where $g \in G$ is a general point depending on $f \in \C(G/H)$.

On the other hand, taking into account the fundamental theorem of tropical algebraic
geometry (cf.~\cite[Theorem~2.1]{maintrop}), Theorem~\ref{prop-trop}
implies that any $G$-invariant discrete valuation $\nu$ comes from
a $\C\{\!\{t\}\!\}$-valued point $x(t)$ of $G/H$ satisfying
\begin{align*}
\ord(f_{ij}(x(t))) = \ord(f_i(x(t)))
\end{align*}
for every $i$ and $j$ and is uniquely determined
by
\begin{align*}
\nu(f_i) = \ord(f_i(x(t))) \text{ and } \nu(g_k) = \ord(g_k(x(t)))
\end{align*}
for every $i$ and $k$.

\section{Proof of the Main Theorem}
\label{two}

We continue to use the notation and the assumptions from the previous section in
this section.

\begin{remark}
\label{rem-linear-action}
The action of $G$ on
\begin{align*}
Z \cong \C^{s_1} \times \ldots \times \C^{s_r} \times \C^* \times \ldots \times \C^*
\end{align*}
is linear on each factor.
\end{remark}

The natural action of the torus $\Spec(\C[\Mm])$ on $Z$ defines a
corresponding $\Mm$-grading on $\C[Z]$. We set $\p \coloneqq \I(G/H) = \ker \Phi$.

\begin{prop}
\label{prop-grading}
The prime ideal $\p$ is $\Mm$-graded, i.e.\ $\glob{G/H}$ is a $\Mm$-graded $\C$-algebra,
if and only if $G/H$ is horospherical.
\end{prop}
\begin{proof}
This follows from \cite[Proposition 7.6]{ti}.
\end{proof}

\begin{lemma}
\label{le-stdemb}
Every non-open $G$-orbit of the spherical embedding
\begin{align*}
G/H \hookrightarrow \overline{G/H} \subseteq Z
\end{align*}
is contained in the closure of a $B$-invariant prime divisor in $G/H$.
\end{lemma}
\begin{proof}
It follows from Proposition~\ref{generators} and the definition
of the map $\Phi$ that $G/H$ and $\overline{G/H}$ have the same ring of global sections.
Therefore the non-open $G$-orbits are at least
of codimension two. It follows from the general theory
of spherical embeddings (cf.~\cite{knopsph}) that
any non-open $G$-orbit
lies in the closure of some $B$-invariant
prime divisor in $G/H$ if there is no $G$-invariant
prime divisor.
\end{proof}

\begin{prop}
Let $Z_i \coloneqq \V(S_{ij}; 1 \le j \le s_i) \subseteq Z$.
Then we have
\begin{align*}
G/H = \overline{G/H} \setminus (Z_1 \cup \ldots \cup Z_r)\text{.}
\end{align*}
\end{prop}
\begin{proof}
Let $O$ be a non-open $G$-orbit in $\overline{G/H}$. By Lemma~\ref{le-stdemb},
there exists a $D_i$ with $O \subseteq \overline{D_i}$. It follows that
$S_{i1}$ vanishes on $O$. Since $V_i$
is an irreducible $G$-module and $O$ is $G$-invariant, the whole module
vanishes on $O$, so $O \subseteq Z_i$ follows.
\end{proof}

We set
\begin{align*}
X_0 \coloneqq Z \setminus (Z_1 \cup \ldots \cup Z_r)\text{.}
\end{align*}
Then $G/H \hookrightarrow X_0$ is a closed embedding, and $G$ and $\T$ act on $X_0$.
Consider the fan $\Sigma_{X_0}$ in $N_\Q$ corresponding to the toric variety
$X_0$. Our plan is to define a fan $\Sigma_X$ extending
the fan $\Sigma_{X_0}$ such that the closure of $G/H$
inside the toric variety $X$ associated
to the fan $\Sigma_X$ is the spherical variety $Y$ corresponding to the
fan $\Sigma$.
We have already established the natural inclusion $\Nm_\Q \hookrightarrow N_\Q$,
so the cones of the fan $\Sigma$ in $\Nm_\Q$ can as well be considered as cones
in $N_\Q$. But extending $\Sigma_{X_0}$ by these one-dimensional cones
is not enough, since it might be impossible to extend the action of $G$ on $X_0$
to the resulting toric variety $X$. The plan can be carried out, however,
if we also add some higher-dimensional cones.

We construct the fan $\Sigma_X$ in $N_\Q$ as follows.
We first define the set
\begin{align*}
\mathfrak{A} \coloneqq \rleft\{\mathfrak{a} \subseteq \rleft\{v_{ij}\rright\}; \text{for each $i$ there is exactly one $j$ with $v_{ij} \notin \mathfrak{a}$}  \rright\}\text{.}
\end{align*}
For each $1 \le l \le n$ and $\mathfrak{a} \in \mathfrak{A}$ we define the cone
\begin{align*}
\sigma_{l, \mathfrak{a}} \coloneqq \cone\rleft(\rleft\{u_l\rright\} \cup \mathfrak{a}\rright)\subseteq N_\Q \text{} 
\end{align*}
and set
\begin{align*}
\Sigma_X \coloneqq \fan\rleft(\{\sigma_{l,\mathfrak{a}}; 1 \le l \le n, \mathfrak{a} \in \mathfrak{A}\}\rright)\text{,}
\end{align*}
the fan generated by the $\sigma_{l,\mathfrak{a}}$ in $N_\Q$.
\begin{prop}
The cones $\sigma_{l,\mathfrak{a}}$ are smooth and compatible.
In particular, the fan $\Sigma_X$ is well-defined.
\end{prop}
\begin{proof}
Let $\sigma_{l,\mathfrak{a}}$ be a cone as defined above. We define the complement
\begin{align*}
\mathfrak{a}^\mathsf{c} \coloneqq \{v_{ij}, w_k : 1\le i \le r, 1\le j \le s_i, 1\le k \le m\} \setminus \mathfrak{a}\text{.}
\end{align*}
We have $N = \langle \mathfrak{a} \rangle \oplus \langle \mathfrak{a}^\mathsf{c} \rangle$.
We write $u_l = u_l' + u_l''$
with $u_l' \in \langle \mathfrak{a} \rangle$ and $u_l'' \in \langle \mathfrak{a}^\mathsf{c} \rangle$.
Then $u_l''$ is primitive, hence $\sigma_{l,\mathfrak{a}}$ is smooth.

Now let $\sigma_{l',\mathfrak{a}'}$ be another cone as defined above, and
let $v \in \sigma_{l,\mathfrak{a}} \cap \sigma_{l',\mathfrak{a}'}$.
We have
\begin{align*}
v &= \sum_{v_{ij} \in \mathfrak{a} \setminus \mathfrak{a}'} d_{ij}v_{ij}
+ \sum_{v_{ij} \in \mathfrak{a} \cap \mathfrak{a}'} d_{ij}v_{ij} + d_0u_l\\
v &= \sum_{v_{ij} \in \mathfrak{a}' \setminus \mathfrak{a}} d'_{ij}v_{ij}
+ \sum_{v_{ij} \in \mathfrak{a} \cap \mathfrak{a}'} d'_{ij}v_{ij} + d'_0u_{l'}
\end{align*}
for some uniquely determined coefficients $d_{ij}, d'_{ij}, d_0, d'_0 \in \Q^+_0$. 
We define
\begin{align*}
m_i = \min \{ \langle v, v_{ij}^* \rangle : 1 \le j \le s_i\}\text{.}
\end{align*}
As $\langle v, v_{ij}^* \rangle = m_i$ for $v_{ij} \notin \mathfrak{a} \cap \mathfrak{a}'$,
the coefficients in the first sum are zero on both lines. It is then clear that
$d_0 = d'_0 = 0$ if $l \ne l'$. It follows that the intersection of the cones
$\sigma_{l,\mathfrak{a}}$ and $\sigma_{l',\mathfrak{a}'}$ is a face of both.
Hence the cones are compatible.
\end{proof}

We denote by $X$ the toric variety associated to the fan $\Sigma_X$.
As the fan $\Sigma_{X_0}$ can also be obtained
by the above construction from the cones $\sigma_{0,\mathfrak{a}}$ where $u_0 \coloneqq 0$, the fan $\Sigma_X$ extends $\Sigma_{X_0}$.
Hence we have an open embedding $X_0 \subseteq X$.
We will now show that the $G$-action on $X_0$ can
be extended to $X$.
We consider $\widehat{N} \coloneqq N \oplus \Z^n$
and regard $N_\Q \subseteq \widehat{N}_\Q \coloneqq \widehat{N} \otimes_\Z \Q$ as naturally included.
We denote the standard basis of $\Z^n$ by
$\{e_l\}_{l=1}^n$. We set
\begin{align*}
\widehat{\sigma}_{l,\mathfrak{a}} &\coloneqq \cone\rleft(\rleft\{e_l\rright\} \cup \mathfrak{a}\rright) \subseteq \widehat{N}_\Q
\end{align*}
and
\begin{align*}
\Sigma_{\widehat{X}} &\coloneqq \fan\rleft(\{\widehat{\sigma}_{l,\mathfrak{a}};  1 \le l \le n, \mathfrak{a} \in
\mathfrak{A}\}\rright)\text{.}
\end{align*}
The associated quasiaffine toric variety $\widehat{X}$ comes with a natural toric
morphism
\begin{align*}
p: \widehat{X} \to X\text{}
\end{align*}
defined by the lattice map $\widehat{N} \to N$ sending $v_{ij} \mapsto v_{ij}$, $w_k \mapsto w_k$, and $e_l \mapsto u_l$.

\begin{remark}
\label{rem-sw}
We have
\begin{align*}
\widehat{X} \cong V^* \times T \times \A^n \setminus \widehat{S}\text{,}
\end{align*}
where $\widehat{S}$ is a closed subset of codimension at least two. 
According to the theory in
\cite{tq}, the toric morphism $p: \widehat{X} \to X$
is a good quotient for the action of a subtorus
$\Gamma \subseteq \T \times (\C^*)^n$.
It is even a geometric quotient. The subtorus $\Gamma$ has
a natural parametrization
$\kappa: (\C^*)^n \xrightarrow{~\cong~} \Gamma$.
Denoting by $W_l$ for $1 \le l \le n$ the coordinates
of $(\C^*)^n$, the action of $\Gamma$ on $\widehat{X}$
is given by
\begin{align*}
\kappa(t) \cdot v = \rleft(\prod_{l=1}^n W_l(t)^{-\rleft\langle u_l, v_{i}^* \rright\rangle}\rright) \cdot v\text{,}
\end{align*}
on $v \in V_i^*$, similarly, by
\begin{align*}
\kappa(t) \cdot v = \rleft(\prod_{l=1}^n W_l(t)^{-\rleft\langle u_l, w_k^* \rright\rangle}\rright) \cdot v\text{,}
\end{align*}
on $v$ in the $k$-th factor $\C^*$ of $T \cong (\C^*)^m$, and finally by the natural action on $\A^n$.

We let $G$ act linearly on the first factors of $\widehat{X}$
with the same action as in
Remark~\ref{rem-linear-action} and trivially on $\A^n$. Then
the action of $G$ on $\widehat{X}$ commutes with the action
of the torus $\Gamma$ from Remark~\ref{rem-sw}. In particular, we obtain an action of $G$ on $X$.
\end{remark}

\begin{remark}
\label{rem-psi}
The natural inclusion $N_\Q \subseteq \widehat{N}_\Q$ defines a $G$-equivariant toric closed embedding
\begin{align*}
\psi: X_0 \to \widehat{X}\text{,}
\end{align*}
and we obtain the following commutative diagram.
\begin{align*}
\xymatrix{
 & \widehat{X} \ar^{p}[rd] \\
X_0 \ar[ru]^\psi \ar@{^{(}->}[rr] & & X
}
\end{align*}
This shows that the action of $G$ on $X$
is an extension of the action of $G$ on $X_0$.
\end{remark}

We denote by $Y'$ the closure of $G/H$ inside $X$. It will
now take some time to prove that $Y'$ is indeed
the spherical embedding of $G/H$ associated to the
fan $\Sigma$, i.e.~that $Y'$ can be identified with $Y$.

\begin{prop}
\label{prop-bigspherical}
The preimage $p^{-1}(G/H)$ with the action of $G \times \Gamma$ is a spherical homogeneous space isomorphic to $G/H \times \Gamma$
with the natural action of $G \times \Gamma$.
The isomorphism
\begin{align*}
\xymatrix{
G/H \times \Gamma \ar[r]^-{\cong} & p^{-1}(G/H)
}
\end{align*}
is given by $(x, t) \mapsto t \cdot \psi(x)$.
\end{prop}

\begin{proof}
As $p$ is a geometric quotient, it follows
that $p^{-1}(G/H) = \Gamma \cdot \psi(G/H)$.
\end{proof}

\begin{remark}
\label{tear-map}
If we denote by $W_l$ for $1 \le l \le n$ the coordinates of $\A^n$
as well as the coordinates of $\Gamma$ under the parametrization $\kappa$ from Remark~\ref{rem-sw}, we
have $\C[\A^n] = \C[W_1, \ldots, W_n]$ and $\C[\Gamma] = \C[W_1^{\pm 1}, \ldots, W_n^{\pm 1}]$.
Considering
the natural inclusion
\begin{align*}
\xymatrix{
p^{-1}(G/H) \ar@{^{(}->}[r] & V^* \times T \times \A^n\text{,}
}
\end{align*}
the isomorphism of Proposition~\ref{prop-bigspherical} is induced by the map
\begin{align*}
\Psi: S(V) \otimes_\C \C[T] \otimes_\C \C[\A^n] &\to \rleft(S(V) \otimes_\C \C[T]\rright)/\p \otimes_\C \C[\Gamma] \\
S_{ij} &\mapsto S_{ij} \otimes  \rleft(\prod_{l=1}^n W_l^{-\rleft\langle u_l, v_{i}^* \rright\rangle}\rright)\\
T_{k} &\mapsto T_{k} \otimes \rleft(\prod_{l=1}^n W_l^{-\rleft\langle u_l, w_k^* \rright\rangle}\rright) \\
W_l &\mapsto W_l\text{.}
\end{align*}
\end{remark}

We denote the closure of $p^{-1}(G/H)$ inside $\widehat{X}$ by $\widehat{Y}$.

\begin{prop}
\label{prop-r}
$\widehat{Y}$ is the quasiaffine spherical embedding of $G/H \times \Gamma$
corresponding to the fan $\widehat{\Sigma}$ in $\rleft(\Nm \oplus \Z^n\rright)_\Q$ which consists of the one-dimensional cones
\begin{align*}
\Q_{\ge 0}(u_l + e_l)
\end{align*}
for $1 \le l \le n$. Furthermore, the ring $\glob{\widehat{Y}}$ is factorial.
\end{prop}
\begin{proof}
Consider the spherical embedding associated to the fan $\widehat{\Sigma}$, which is quasiaffine
with factorial ring of global sections (cf.~\cite[Proposition~4.1.1]{brcox}).
By possibly adding some
orbits
of codimension at least two,
we can make it affine (cf.~\cite[Proof
of Theorem 6.7]{knopsph})
and denote its coordinate ring by $R$. We will show that it is
the closure of $p^{-1}(G/H)$ inside the affine toric variety 
$V^* \times T \times \A^n$.
Consider the following diagram.
\begin{align*}
\text{{\small $
\xymatrix{
S(V) \otimes_\C \C[T] \otimes_\C \C[\A^n] \ar[r] & R \ar@{^{(}->}[d]\\
& \rleft(S(V) \otimes_\C \C[T]\rright)/\p \otimes_\C \C[\Gamma]
\ar@{=}[r] & \glob{G/H \times \Gamma}
}
$}}
\end{align*}
As the ring $R$ is normal, it consists of all the elements in the
ring of global sections of
$G/H \times \Gamma$ which have
a nonnegative value under the valuations induced by
the $G$-invariant
prime divisors in $\Spec(R)$.
It follows that the image of $\Psi$ (cf.~Remark~\ref{tear-map})
is contained in $R$.

Additionally, all the $B\times \Gamma$-eigenvectors in $R$,
and even the $G \times \Gamma$-modules generated by them, belong
to the image of the horizontal map, so we obtain
surjectivity of the horizontal map with the same argument as
in Proposition~\ref{generators}.
The composition of the horizontal map and the vertical localization is the same
map as in Proposition~\ref{prop-bigspherical}. It follows that we have
$\overline{p^{-1}(G/H)} = \Spec(R)$ inside 
$V^* \times T \times \A^n$. 

We use the same symbols for elements
of $S(V) \otimes_\C \C[T] \otimes_\C \C[\A^n]$
and their images in $R$.
As the valuations induced
by the $G$-invariant prime divisors in $\Spec(R)$
have values of $0$ or $1$ respectively on the 
$W_l$ and of $0$
on the $S_{ij}$, it follows that
the $W_l$ and the $S_{ij}$ 
are pairwise nonassociated prime elements of $R$.
The $G$-orbits which lie in the closure of $D_i \times \Gamma$
are contained in $\V(S_{ij}; 1 \le j \le s_i)$, while
the other $G$-orbits of codimension at least two are contained in $\V(W_{l_1}) \cap \V(W_{l_2})$
for some $l_1 \ne l_2$. 
Therefore intersecting with
$\widehat{X}$, i.e.~removing the set $\widehat{S}$, removes exactly the $G$-orbits
of codimension at least two, and the result follows.
\end{proof}

We will reuse the following fact from the proof
of Proposition~\ref{prop-r}.

\begin{remark}
\label{re-primes}
The $W_l$ and the $S_{ij}$ are pairwise
nonassociated prime elements of $\glob{\widehat{Y}}$.
\end{remark}

\begin{remark}
\label{rem:norm}
The restriction
\begin{align*}
p|_{\widehat{Y}}: \widehat{Y} \to Y'
\end{align*}
is a good geometric quotient. In particular, $Y'$ is normal.
\end{remark}

In the following Lemma~\ref{le-val}, we will assume that $\Sigma$ contains
only one one-dimensional cone $\sigma$. We also allow
the cone $\sigma$ to lie outside the valuation cone $\Vm$. In that
case, the spherical variety $Y$ is not defined, but the closure
$Y'$ of $G/H$ inside $X$ can still be constructed. 
Remark~\ref{rem:norm} not being available, the
normality of $Y'$ is not certain.
This more general setting will only be required for one direction of the proof
of Theorem~\ref{prop-trop} at the end of this section.

\begin{lemma}
\label{le-val}
Any irreducible component of $Y' \setminus (G/H)$
which intersects the $\T$-orbit in $X$ corresponding to $\sigma$ induces a
$G$-invariant valuation which induces $\sigma$ (after possibly normalizing $Y'$).
\end{lemma}
\begin{proof}
We consider the affine open toric subvariety $U_\sigma$ of $X$ with two
orbits, the torus and the orbit corresponding to the cone $\sigma$.
Let $\pi \in \C[U_\sigma]$ be the prime element such that
$\V(\pi)$ is the orbit corresponding to the cone $\sigma$.
We
obtain the following commutative diagram.
\begin{align*}
\xymatrix{
\C[U_\sigma] \ar@{^{(}->}[rrr]\ar[d]  &&& \rleft(S(V) \otimes_\C \C[T]\rright)_{\p} \ar[d]\\
\C[U_\sigma] \big/ \rleft(\p\cap\C[U_\sigma]\rright) \ar@{^{(}->}^-{\Nf}[r] 
& R_1 \ar@{^{(}->}^-{\Lf}[r] & R_2 \ar@{^{(}->}[r] & \rleft(S(V) \otimes_\C \C[T]\rright)_{\p} \big/ \p  \ar@{=}[r] & \C\rleft(G/H\rright)
}
\end{align*}
It is sufficient to localize at the prime ideal $\p$ to get the
inclusion in the first row since the $S_{ij}$ are not in $\p$.
In the bottom row, $\Nf$ denotes normalization, and $\Lf$
is localization in such a way that $\V(\pi) = \V(\pi_0)$ in $\Spec(R_2)$ 
where $\pi_0 \in R_2$ is a prime element.
Each $L \in \{S_{ij}, T_k\}$ can be written as
$L = c_1\pi^{d_1}/\pi^{d_2}$ with $c_1 \in \C[U_\sigma]^*$ and $d_1, d_2 \in \N_0$
since there are no $\T$-invariant prime divisors other than $\V(\pi)$ in $U_\sigma$.
Therefore we have $L = c_2\pi_0^{d_3d_1}/\pi_0^{d_3d_2}$ for some
$c_2 \in R_2^*$ and $d_3 \in \N_0$. It follows that $\V(\pi_0)$ induces
the correct cone $\sigma$.
\end{proof}

\begin{prop}
\label{prop-main}
$Y'$ is the spherical
embedding of $G/H$ corresponding to the fan $\Sigma$.
\end{prop}

\begin{proof}
As general embeddings of $G/H$ can be obtained by gluing embeddings
with only one closed orbit, we may assume $n = 1$.
In this case, it is clear that $Y'$ has two $G$-orbits,
and Lemma~\ref{le-val} implies that the closed $G$-orbit induces
the correct $G$-invariant valuation.
\end{proof}

From now on, we will identify $Y'$ and $Y$.
We denote by $X_1, \ldots, X_n$
the $\T$-invariant prime divisors in $X$ corresponding
to $u_1, \ldots, u_n$. Note that $\Gamma(Y, \Om^*_Y) = \C^*$
implies $\Gamma(X, \Om^*_X) = \C^*$.

We temporarily reuse our notation in a more general setting to
give an overview over the next step.
Let $X$ be a toric $\T$-variety with $\Gamma(X, \Om^*_X) = \C^*$ and
$\T$-invariant prime divisors $X_1, \ldots, X_n$.
Then we have the canonical quotient construction
$\pi: \widetilde{X} \to X$, where $\widetilde{X}$ is a quasiaffine
toric variety and $\Rm(X) \cong \glob{\widetilde{X}}$ (cf.~\cite[Theorem 5.1.11]{cls}).
Let $\iota: Y \hookrightarrow X$ be a closed embedding with
$\Gamma(Y, \Om^*_Y) = \C^*$ as well as $X$ and $Y$ smooth. It follows from
the work of Hausen (cf.~\cite{ha}) that $\Rm(Y) \cong \glob{\pi^{-1}(Y)}$ if
$X_l \cap Y$ is an irreducible hypersurface in $Y$ 
intersecting the $\T$-orbit which is dense in $X_l$ for each $1 \le l \le n$
and the map $\iota^*: \Cl(X) \to \Cl(Y)$ induced by
the pullback of Cartier divisors is an isomorphism.
Such an embedding is called \em neat\em.

We now return to our setting. We identify
the prime divisors $D_i$ in $G/H$ and $\V(S_{ij})$ in $X_0$
with their closures inside $Y$ and $X$ respectively
as long as there is no danger of confusion.

\begin{prop}
\label{prop-neat}
The closed embedding
\begin{align*}
\iota: Y \hookrightarrow X
\end{align*}
is a neat embedding in the sense of \cite[Definition 2.5]{ha}.
\end{prop}

\begin{proof}
Consider the following commutative diagram.
\begin{align*}
\xymatrix{
\widehat{Y} \ar@{^{(}->}^{\widehat{\iota}}[r] \ar[d]_{p} & \widehat{X} \ar[d]^{p} \\
Y \ar@{^{(}->}_{\iota}[r] & X 
}
\end{align*}
Since the diagram commutes, $\iota^{-1}(X_l)$ and $\iota^{-1}(\V(S_{ij}))$ are irreducible for each $i$, $j$, and $l$,
and using Remark~\ref{re-primes} as in the proof
of Theorem~\ref{prop-trop}, we obtain that they intersect the corresponding
$\T$-orbit of codimension one in $X$.

We have a pullback map of Cartier divisors $\iota^*: \Div(X) \to \Div(Y)$, and clearly
$\Supp(\iota^*(X_l)) \subseteq Y_l$ and $\Supp(\iota^*(\V(S_{i1}))) \subseteq D_i$.
Using the diagram again, we see that locally the pullbacks of $X_l$
and $\V(S_{i1})$ are prime divisors. Therefore we have
$\iota^*(X_l) = Y_l$ and $\iota^*(\V(S_{i1})) = D_i$.

Using the explicit
descriptions of
the divisor class groups of toric and spherical varieties
(cf.~\cite[Theorem~4.1.3]{cls} and \cite[Proposition~4.1.1]{brcox}),
we obtain that the induced pullback map
$\iota^*: \Cl(X) \to
\Cl(Y)$
is an isomorphism.
\end{proof}

We are now almost ready to describe the Cox ring $\Rm(Y)$.
We recall the homogenization operation
from the introduction. The first step is the map $\alpha: \C[Z] \to \big(\C[Z]\big)[W_1, \ldots, W_n]$
defined as follows.
For each $f \in \C[Z]$ and $u \in \Nm$ we define
\begin{align*}
\ord_u(f) \coloneqq \min_{\mu\in\Mm} \rleft\{\rleft\langle u, \mu \rright\rangle; f^{(\mu)} \ne 0\rright\}\text{,}
\end{align*}
and set
\begin{align*}
f^\alpha \coloneqq \frac{\sum_{\mu \in \Mm} \rleft(f^{(\mu)} \prod_{l=1}^n W_l^{\rleft\langle u_l, \mu \rright\rangle}\rright)}
{\prod_{l=1}^n W_l^{\ord_{u_l}(f)}}\text{.}
\end{align*}
The second step is the map $\beta: \big(\C[Z]\big)[W_1, \ldots, W_n]
\to S(V)[W_1, \ldots, W_n]$ sending
$T_k \mapsto 1$ for each $1 \le k \le m$.
Finally, we define the map $h: \C[Z] \to S(V)[W_1, \ldots, W_n]$
by composing the two steps, i.e.~$h \coloneqq \beta \circ \alpha$.

\begin{theorem}
\label{th-cox}
We have
\begin{align*}
\Rm(Y) \cong S(V)[W_1, \ldots, W_n]\rleft/\middle(f^h; f \in \p\rright)\text{,}
\end{align*}
with $\Cl(Y)$-grading given by $\deg(S_{ij}) = [D_i]$ and $\deg(W_l) = [Y_l]$.
\end{theorem}

\begin{proof}
Let $\pi: \widetilde{X} \to X$ be the canonical quotient construction.
We set $\widetilde{Y} \coloneqq \pi^{-1}(Y)$. By \cite[Construction~5.1.4]{coxrings},
the scheme-theoretic fiber $\widetilde{X} \times_X Y$ is reduced. 
There is a canonical toric closed embedding
$\phi: \widetilde{X} \hookrightarrow \widehat{X}$
such that $\pi = p \circ \phi$.
Note that $\phi^*$
sends $f \mapsto f^\beta$. The assertion now follows from Proposition~\ref{prop-neat} and
\cite[Theorem 2.6]{ha} if we show that the ideal $\widehat{\p}$
of $\widehat{Y}$ in $\widehat{X}$ is
$\p^\alpha \coloneqq \rleft(f^\alpha; f \in \p\rright)$.

We use the map $\psi$ from Remark~\ref{rem-psi}. We have $\widehat{\p} = \I(\Gamma \cdot \psi(\V(\p)))$. Every $f^\alpha \in \p^\alpha$ vanishes
on $\psi(\V(\p))$ since $\psi^*(f^\alpha) = f \in \p$, and $f^\alpha$ is a $\Gamma$-eigenvector,
so $f^\alpha \in \widehat{\p}$, and $\p^\alpha \subseteq \widehat{\p}$ follows. Now,
let $g \in \widehat{\p}$. As $\widehat{\p}$ is a homogeneous ideal with
respect to the $\X(\Gamma)$-grading, all
homogeneous components $g^{(\xi)}$ are in $\widehat{\p}$.
It is not difficult to see that
$g^{(\xi)} = \rleft(\psi^*\rleft(g^{(\xi)}\rright)\rright)^\alpha \prod_{l=1}^n W_l^{d_l}$
for some exponents $d_l \in \N_0$. Since $\psi^*\rleft(g^{(\xi)}\rright) \in \p$, the inclusion
$\widehat{\p} \subseteq \p^\alpha$ follows.
\end{proof}

\begin{remark}
If $\Gamma(G/H, \Om^*_{G/H}) = \C^*$, we have $\widetilde{X} = \widehat{X}$, and $\phi$ is the identity.
\end{remark}

Cox rings are always factorially graded (cf.~\cite[Theorem~7.3]{bhhom}),
but not factorial in general (cf.~\cite[Example~4.2]{cfact}).
A sufficient condition for the Cox ring to be factorial
is the divisor class group being free (cf.~\cite[Proposition~8.4]{bhhom}, also \cite[Corollary~1.2]{ekw}).
The last part of the Main Theorem also provides a sufficient condition for the Cox ring
of a spherical variety to be factorial.

\begin{theorem}
\label{thm-fact-fact}
If $H$ is connected, $\Rm(Y)$ is a factorial ring.
\end{theorem}
\begin{proof}
The finitely generated free abelian group $\Gamma(G/H, \Om^*_{G/H})/\C^*$
is naturally isomorphic to the subgroup $\X(C)^H \subseteq \X(C)$
consisting of $H$-invariant characters. The quotient group
$\X(C) / \X(C)^H$ is free 
as $H$ is connected.
Therefore there exists a decomposition $\X(C) = \X(C)^H \oplus \X(C)^\dagger$, and
we can choose the $f_{ij}$ in such a way that $C$ acts on the $f_{ij}$ with characters belonging to $\X(C)^\dagger$.
It follows that for each $1 \le k \le m$ we have a one-parameter subgroup of $C$ acting nontrivially
only on the variable $T_k$. As $\widehat{Y}$ is $C$-invariant, we obtain $\widehat{Y} \cong \widetilde{Y} \times T$.
Therefore the factoriality of $\Rm(Y) \cong \glob{\widetilde{Y}}$ follows
from the factoriality of $\glob{\widehat{Y}}$ (cf.~Proposition~\ref{prop-r}).
\end{proof}

Finally, we provide the proof of Theorem~\ref{prop-trop}.
 
\begin{proof}[Proof of Theorem~\ref{prop-trop}]
We have to show $\Vm = \trop(G/H \cap \T) \cap \Nm_\Q$.
Using \cite[Lemma 2.2]{tev},
we get the inclusion from the left to the right
using Proposition~\ref{prop-main} if
$Y_l$ intersects
the $\T$-orbit which is dense in $X_l$.
This is the case, since otherwise it would
follow from Remark~\ref{re-primes} that
the codimension of $Y_l$ in $Y$ is at least two.
If the inclusion from the right to the left did not hold,
Lemma~\ref{le-val} would yield a
non-existing $G$-invariant valuation.
\end{proof}

\section{Generalization to arbitrary spherical homogeneous spaces}
\label{three}
We now consider the case where the spherical homogeneous space is allowed
to have nontrivial divisor class group.
Let $\arb{G}$ be a connected reductive group and $\arb{H}
\subseteq \arb{G}$ a spherical subgroup. We may again assume
that $\arb{G}$ is of simply connected type,
i.e.~$\arb{G} = \arb{G}^{ss} \times \arb{C}$
where $\arb{G}^{ss}$ is semisimple simply connected and $\arb{C}$ is a torus.

We fix a Borel subgroup $\arb{B} \subseteq \arb{G}$ such that
the base point $1 \in \arb{G}/\arb{H}$ lies in the open $\arb{B}$-orbit and denote
by $\arb{\Dm} \coloneqq \rleft\{\arb{D}_1, \ldots, \arb{D}_r\rright\}$ the set of $\arb{B}$-invariant
prime divisors in $\arb{G}/\arb{H}$.
For each $\arb{D}_i$ the pullback under the quotient map
$\arb{G} \to \arb{G}/\arb{H}$ is a divisor
with equation $\arb{f}_i \in \C\rleft[\arb{G}\rright]$ where $\arb{f}_i$ is uniquely determined
by being $\arb{C}$-invariant with respect to the action from the left
and $\arb{f}_i(1) = 1$ (cf.~\cite[4.1]{brcox}).
The group $\arb{H}$ acts from the right on each $\arb{f}_i$ with a character
$\arb{\chi}_i \in \X\rleft(\arb{H}\rright)$. 

We define
\begin{align*}
G &\coloneqq \arb{G} \times (\C^*)^{\arb{\Dm}} \\
H &\coloneqq \rleft\{\rleft(h, \arb{\chi}_1(h), \ldots, \arb{\chi}_r(h)\rright); h \in \arb{H}\rright\} \subseteq G\text{,}
\end{align*}
and set $B \coloneqq \arb{B} \times (\C^*)^{\arb{\Dm}}$.
We have a quotient map $\arb{\pi}: G/H \to \arb{G}/\arb{H}$, which is
a good geometric quotient by the torus $(\C^*)^{\arb{\Dm}}$.
There is a natural isomorphism $H \cong \arb{H}$, and $G/H$ is a spherical homogeneous space.
The pullbacks of the $\arb{B}$-invariant prime divisors in
$\arb{G}/\arb{H}$ under the quotient map $\arb{\pi}$
are exactly the $B$-invariant prime divisors
$D_1, \ldots, D_r$ in $G/H$.

We denote by $\Pic_{\arb{G}}\rleft(\arb{G}/\arb{H}\rright)$ the group
of isomorphism classes of $\arb{G}$-linearized invertible sheaves
on $\arb{G}/\arb{H}$. The character lattice of $(\C^*)^{\arb{\Dm}}$
is $\Z^{\arb{\Dm}}$ with standard basis $\{\eta_1, \ldots, \eta_r\}$.

\begin{prop}
\label{prop-ex-res}
The following diagram of natural maps is commutative, and the top row is an exact sequence.
\begin{align*}
\xymatrix{
0 \ar[r] & \arb{\Mm} \ar[r] & \X\rleft(\arb{C}\rright) \oplus \Z^{\arb{\Dm}} \ar[r] \ar[d]_{\cong} &
\Pic_{\arb{G}}\rleft(\arb{G}/\arb{H}\rright) \ar[r] & 0 \\
& & \X\rleft(G\rright) \ar[r]_{\operatorname{res}} & \X\rleft(H\rright) \ar[u]_{\cong}
}
\end{align*}
\end{prop}

\begin{proof}
The map $\X\rleft(\arb{C}\rright) \oplus \Z^{\arb{\Dm}} \to
\Pic_{\arb{G}}\rleft(\arb{G}/\arb{H}\rright)$
sends $\chi \mapsto \Om_{\arb{G}/\arb{H}}(\chi)$
for $\chi \in \X(\arb{C})$
and $\eta_i \mapsto \Om_{\arb{G}/\arb{H}}(\arb{D}_i)$
as in \cite[Proposition~4.1.1]{brcox}
where $\Om_{\arb{G}/\arb{H}}(\arb{D}_i)$ is canonically
$\arb{G}$-linearized as in \cite[4.1]{brcox}.
The left-hand isomorphism is obvious,
$\operatorname{res}$ is the restriction map,
and the 
right-hand isomorphism sends $\chi \mapsto \mathscr{L}_{\arb{G}/\arb{H}}(-\chi)$ for
$\chi \in \X(H)$ where
$\mathscr{L}_{\arb{G}/\arb{H}}(-\chi)$
is the standard construction (cf.~\cite[2.1]{local} or \cite[after Proposition~2.4]{ti}).
It is not difficult to see that the diagram commutes.
We denote by $\orho\rleft(\arb{D}_i\rright) \in \arb{\Nm}_\Q$ the vector corresponding to 
the restriction of the discrete valuation induced by $\arb{D}_i \in \arb{\Dm}$
and set $d_i(\mu) \coloneqq \rleft(\orho\rleft(\arb{D_i}\rright)\rright)(\mu) \in \Z$.
Then the map $\arb{\Mm} \to \X\rleft(\arb{C}\rright) \oplus \Z^{\arb{\Dm}}$ sends
\begin{align*}
\mu \mapsto \sum_{i=1}^r d_i(\mu) \eta_i -
\mu|_{\arb{C}}\text{.}
\end{align*}
If $d_i(\mu) = 0$ for all $1 \le i \le r$ and $\mu$ is the $\arb{B}$-weight
of $f \in \C\rleft(\arb{G}/\arb{H}\rright)^{\rleft(\arb{B}\rright)}$, it follows that $\operatorname{div}(f) = 0$, so $f$ is
a unit in $\glob{\arb{G}/\arb{H}}$. This means we have $f \in 
\X\rleft(\arb{C}\rright)$, so $\mu|_{\arb{C}} = 0$ if and only if $\mu = 0$.
Therefore the map is injective. Using the first statement of \cite[Proposition~4.1.1]{brcox},
which in fact does not require the assumption $\glob{\arb{G}/\arb{H}} = \C$,
we obtain the exactness of the top row.
\end{proof}

\begin{cor}
The spherical homogeneous space $G/H$ has trivial divisor class group.
\end{cor}

\begin{proof}
By a theorem of Popov, $\Cl(G/H) = \Pic(G/H)$ is trivial if and only if every character of $H$
is the restriction of a character of $G$ (cf.~\cite{Popov:pic}, \cite[Theorem~2.5]{ti}).
By Proposition~\ref{prop-ex-res}, $\operatorname{res}: \X(G) \to \X(H)$ is surjective.
\end{proof}

We will continue to use the notation from the previous sections
for the spherical homogeneous space $G/H$.
Where applicable, we use the same notation
for the spherical homogeneous space $\arb{G}/\arb{H}$,
but in boldface symbols. In particular,
$\arb{\Mm}$ is the weight lattice of $\arb{B}$-eigenvectors
in the function field $\C\rleft(\arb{G}/\arb{H}\rright)$, $\arb{\Nm} = \Hom\rleft(\arb{\Mm}, \Z\rright)$
the dual lattice, $\arb{\Nm}_\Q = \arb{\Nm} \otimes_\Z \Q$,
and $\arb{\Vm}$ is the valuation cone of $\arb{G}/\arb{H}$.

We fix a convenient choice for
the prime elements $f_1, \ldots, f_r \in \glob{G/H}$.
We have a natural inclusion $\epsilon: \Z^{\arb{\Dm}} \hookrightarrow \C[G]^*$, $\chi \mapsto \epsilon^\chi$. Then $f_i \coloneqq
\arb{f}_i\epsilon^{-\eta_i} \in \C[G]$ is $H$-invariant under the action from the right, and
$f_i$ is a prime element in $\glob{G/H}$ with $\V(f_i) = D_i$.
The torus $(\C^*)^{\arb{\Dm}}$ acts on $f_i \in \glob{G/H}$ with weight $\eta_i$.

Recall the decomposition $\Mm = \Mm_\hi \oplus \Mm_\un$. We have a corresponding
decomposition of the dual lattice $\Nm = \Nm_\hi \oplus \Nm_\un$.

We consider a spherical embedding $\arb{G}/\arb{H} \hookrightarrow \arb{Y}$
which contains only non-open $\arb{G}$-orbits of codimension one,
given by a fan $\arb{\Sigma}$ in $\arb{\Nm}_\Q$,
and assume $\Gamma\rleft(\arb{Y}, \Om_{\arb{Y}}^*\rright) = \C^*$. We denote
by $\arb{Y}_1, \ldots, \arb{Y}_n$ the $\arb{G}$-invariant prime
divisors in $\arb{Y}$.

The next two results will allow us to obtain
a fan $\Sigma$ in $\Nm_\Q$ with associated spherical
embedding $G/H \hookrightarrow Y$ from the fan $\arb{\Sigma}$ in $\arb{\Nm}_\Q$.
We will then construct the ring $\Rm(Y)$ exactly as in Section~\ref{two}. As
we may have $\Gamma(Y, \Om_Y^*) \ne \C^*$, the ring $\Rm(Y)$ may not be the Cox ring of $Y$.
This ring will, however, be the Cox ring $\Rm\rleft(\arb{Y}\rright)$ of $\arb{Y}$.

The quotient map $\arb{\pi}: G/H \to \arb{G}/\arb{H}$ induces an inclusion
$\arb{\pi}^*: \arb{\Mm} \hookrightarrow \Mm$,
which is the restriction of the natural inclusion
$\arb{\pi}^+: \X\rleft(\arb{B}\rright) \hookrightarrow \X(B)$.
We also obtain the surjective dual map
$\arb{\pi}_*: \Nm_\Q \to \arb{\Nm}_\Q$.

\begin{prop}
\label{prop-valcone}
We have
\begin{align*}
\Vm = \arb{\pi}_*^{-1}\rleft(\arb{\Vm}\rright)\text{.}
\end{align*}
\end{prop}
\begin{proof}
Let $v \in \Nm_\Q$ with $\arb{\pi}_*(v) = 0$.
Interpreting $v$ as map $v: \Mm \to \Q$, this means $v \circ \arb{\pi}^* = 0$.
Therefore there exists an extension $v^+: \X(B) \to \Q$ of $v$ with $v^+ \circ \arb{\pi}^+ = 0$.
By \cite[Corollary~5.3]{knopsph}, we obtain $v \in \Vm$, therefore $\arb{\pi}_*^{-1}(0) \subseteq
\Vm$. Finally, we use \cite[Corollary~1.5]{knopsph}.
\end{proof}

As the $B$-weights of the kernel of the restriction map $\operatorname{res}: \X(G) \to \X(H)$ are exactly
the lattice $\Mm_\un$, Proposition~\ref{prop-ex-res} yields an isomorphism
$\gamma: \arb{\Mm} \to \Mm_\un$.

\begin{prop}
\label{prop-composedmap}
The restricted map
\begin{align*}
\arb{\pi}_*|_{\rleft(\Nm_\un\rright)_\Q}: \rleft(\Nm_\un\rright)_\Q \to \arb{\Nm}_\Q
\end{align*}
is dual to the isomorphism $\gamma: \arb{\Mm} \to \Mm_\un$. In particular,
it is itself an isomorphism.
\end{prop}
\begin{proof}
We have to show that for each $v \in \Nm_\Q$ with $v|_{\Mm_\hi} = 0$
and each $\mu \in \arb{\Mm}$
we have $v(\gamma(\mu)) = v(\arb{\pi}^*(\mu))$. It therefore
suffices to show that $\gamma(\mu) - \arb{\pi}^*(\mu) \in \Mm_\hi$ for
each  $\mu \in \arb{\Mm}$.

Let $\mu \in \arb{\Mm}$ be the $\arb{B}$-weight of the
$\arb{B}$-eigenvector
$f \in \C\rleft(\arb{G}/\arb{H}\rright) \subseteq \C(\arb{G})$.
Using the notation from the proof of Proposition~\ref{prop-ex-res}, we
necessarily have
\begin{align*}
f = c\prod_{i=1}^r \arb{f}_i^{d_i(\mu)}\text{,}
\end{align*}
where $c \in \C\rleft[\arb{G}\rright]^*$ has left $\arb{B}$-weight $\mu|_{\arb{C}}$. Therefore we obtain
\begin{align*}
\gamma(\mu) = -\sum_{i=1}^r d_i(\mu)\eta_i + \mu|_{\arb{C}} && \text{ and } &&
\arb{\pi}^*(\mu) = \mu|_{\arb{C}} + \sum_{i=1}^r d_i(\mu)\omega_i
\end{align*}
where $\omega_i$ is the left
$\arb{B}$-weight of $\arb{f}_i$, hence $\gamma(\mu) - \arb{\pi}^*(\mu) \in \Mm_\hi$.
\end{proof}

We obtain the fan $\Sigma$ in $\Nm_\Q$
(with associated spherical embedding $G/H \hookrightarrow Y$)
as preimage under $\arb{\pi}_*|_{\rleft(\Nm_\un\rright)_\Q}$ of the fan $\arb{\Sigma}$ in $\arb{\Nm}_\Q$.

\begin{remark}
There is a good geometric quotient of the whole toric variety $X$
by the action of $(\C^*)^{\arb{\Dm}}$. This means that $\arb{\pi}$ can be
extended to a quotient $\arb{\pi}: X \to \arb{X}$.
We obtain the following natural commutative diagram.
\begin{align*}
\xymatrix{
\rleft(\Nm_\un\rright)_\Q \ar@{^{(}->}[r] & \Nm_\Q \ar[r]^{\arb{\pi}_*} \ar@{^{(}->}[d] & \arb{\Nm}_\Q \ar@{^{(}->}[d] \\
& N_\Q \ar[r]^{\arb{\pi}_*} & \arb{N}_\Q
}
\end{align*}
Similarly to Section~\ref{two}, we obtain $\arb{Y}$ as
closure of $\arb{G}/\arb{H}$ inside $\arb{X}$, and the embedding $\arb{Y} \hookrightarrow \arb{X}$ is neat.
\end{remark}

\begin{theorem}
\label{th-cox-gen}
We have
\begin{align*}
\Rm\rleft(\arb{Y}\rright) \cong S(V)[W_1, \ldots, W_n]\rleft/\middle(f^h; f \in \p\rright)\text{,}
\end{align*}
with $\Cl\rleft(\arb{Y}\rright)$-grading given by $\deg(S_{ij}) = \rleft[\arb{D}_i\rright]$ and $\deg(W_l) = \rleft[\arb{Y}_l\rright]$.
\end{theorem}
\begin{proof}
As in Section~\ref{two}, we construct the good quotient
$p: \widehat{X} \to X$, and
we have the canonical quotient construction
$\pi: \widetilde{X} \to \arb{X}$
as well as a canonical toric closed embedding
$\phi: \widetilde{X} \to \widehat{X}$
such that $\pi = \arb{\pi} \circ p \circ \phi$. The result
now follows as in Section~\ref{two}.
\end{proof}

\begin{theorem}
If $\arb{H}$ is connected, $\Rm\rleft(\arb{Y}\rright)$ is a factorial ring.
\end{theorem}
\begin{proof}
We embed $Y$ into another toric variety $X'$ using primes $f'_i \in \glob{G/H}$ 
satisfying the requirements of the proof of Theorem \ref{thm-fact-fact} instead of the primes $f_i$.
The $f'_i$ can be chosen in such a way that for each $1 \le i \le r$ there is a unit $c_i$
such that $f'_i = c_if_i$ and $c_i$ is the product of elements of $\{g_k\}_{k=1}^m$.
As a basis for the $G$-module $\langle G\cdot f'_i\rangle$ we then choose
$\{c_if_{ij}\}_{j=1}^{s_i}$. In this case, there is a toric isomorphism
$X' \cong X$ which fixes $G/H$ and therefore $Y$ as well.
We obtain the following commutative diagram.
\begin{align*}
\xymatrix{
& \widehat{X}' 
\ar[r]^{p'} & X' \ar[d]^\cong  \\
\widetilde{X} \ar[r]_{\phi} \ar[ru]^{\phi'} & \widehat{X} \ar[r]_p &
X \ar[r]_{\arb{\pi}} & \arb{X}
}
\end{align*}
The factoriality of $\Rm\rleft(\arb{Y}\rright)$ now follows as in Theorem \ref{thm-fact-fact}.
\end{proof}

The following Theorem~\ref{th-cox-hor} follows directly
from \cite[Theorem~4.3.2]{brcox}. In order to be self-contained,
we give another proof.

\begin{theorem}
\label{th-cox-hor}
If $\arb{G}/\arb{H}$ is horospherical, we have
\begin{align*}
\Rm\rleft(\arb{Y}\rright) \cong \Rm\rleft(\arb{G}/\arb{P}\rright)[W_1, \ldots W_n]\text{,}
\end{align*}
where $\arb{P} \coloneqq N_{\arb{G}}\rleft(\arb{H}\rright)$.
\end{theorem}

\begin{proof}
As $\arb{G}/\arb{H}$ being horospherical implies
that $G/H$ is horospherical as well, the ideal $\p$ is $\Mm$-graded
by Proposition \ref{prop-grading}. This means,
for each $f \in \p$ and $\mu \in \Mm$ we
have $f^{(\mu)} \in \p$. As $\rleft(f^{(\mu)}\rright)^\alpha = f^{(\mu)}$,
the ideal $(f^h; f \in \p)$ can be generated by elements which do not
contain any of the variables $W_l$. It follows that
$\Rm\rleft(\arb{Y}\rright) \cong R[W_1, \ldots, W_n]$ for some ring $R$
which depends only on the homogeneous space.

It only remains to show that $R \cong \Rm\rleft(\arb{G}/\arb{P}\rright)$.
We have canonical maps
\begin{align*}
\xymatrix{
\arb{G} \ar[r] & \arb{G}/\arb{H} \ar[r]^-{\cong} & 
\rleft(\arb{G} \times \arb{T}\rright)/\arb{P} \ar[r] & \arb{G}/\arb{P}\text{,}
}
\end{align*}
where $\arb{T} \coloneqq \arb{P}/\arb{H}$ is a torus, 
$\arb{P}$ acts on $\arb{T}$ via $\arb{P} \to \arb{T}$, and $\arb{P}$ acts
on $\arb{G} \times \arb{T}$ via $p \cdot (g, t) \coloneqq (gp^{-1}, pt)$.
The last map has fibers isomorphic to $\arb{T}$ and is a trivial fibration
over the open orbit of any Borel subgroup of $\arb{G}$ (cf.~\cite[after Theorem~2.2]{bm}).
The pullbacks of the $\arb{B}$-invariant prime divisors in $\arb{G}/\arb{P}$
are exactly the $\arb{B}$-invariant prime divisors in $\arb{G}/\arb{H}$.
In particular, $\arb{P}$ acts from the right on each $\arb{f}_i \in \C\rleft[\arb{G}\rright]$
with an extension of the character $\arb{\chi}_i \in \X\rleft(\arb{H}\rright)$ which we
will also call $\arb{\chi}_i \in \X\rleft(\arb{P}\rright)$. We define
\begin{align*}
P &\coloneqq \rleft\{\rleft(p, \arb{\chi}_1(p), \ldots, \arb{\chi}_r(p)\rright); p \in \arb{P}\rright\} \subseteq G\text{.}
\end{align*}
It is not difficult to see that $G/H$ is isomorphic to $\rleft(G \times \arb{T}\rright)/P$
where $P \cong \arb{P}$ acts on $G \times \arb{T}$ via $p \cdot (g, t) \coloneqq (gp^{-1}, pt)$.
As all characters of $P$ can be extended to $G$ (cf.~Proposition \ref{prop-ex-res}), we
obtain $G/H \cong G/P \times \arb{T}$, and the result follows.
\end{proof}

\section{Examples}
\label{four}

\begin{example}
Consider $G \coloneqq \SL(3)$ and $H \coloneqq U$, the set of unipotent
upper triangular matrices. 
Let $B \subseteq G$ be the Borel subgroup
of upper triangular matrices, and let $G$ act linearly on $\C^3 \times \C^3$ by acting
naturally on the first factor and with the contragredient action on the second
factor. We denote the coordinates of the first factor by $S_{11}, S_{12}, S_{13}$
and the coordinates of the second factor by $S_{21}, S_{22}, S_{23}$. Then the point
$((1,0,0), (0,0,1))$ has isotropy group $U$, and its orbit is
$\V(S_{11}S_{21}+S_{12}S_{22}+S_{13}S_{23})$. The homogeneous space
$G/H$ is horospherical with $B$-invariant prime divisors $D_1 \coloneqq \V(S_{13})$ and $D_2 \coloneqq \V(S_{21})$
and has trivial divisor class group. Consider the embedding $Y$ corresponding
to a fan containing exactly $n$ one-dimensional cones.
By Theorem~\ref{th-cox-hor} and setting $P\coloneqq N_G(H)$, we obtain
\begin{align*}
\Rm(G/P) &\cong \C[S_{11}, S_{12}, S_{13}, S_{21}, S_{22}, S_{23}] \rleft/ \middle(S_{11}S_{21}+S_{12}S_{22}+S_{13}S_{23}\rright) \\
\Rm(Y) &\cong \Rm(G/P)[W_1, \ldots, W_n]\text{.}
\end{align*}
\end{example}

\begin{example}
\label{ex:comp}
Consider $G \coloneqq \SL(d)$ for $d \ge 3$ and $H \coloneqq \SL(d-1)$ embedded
as the lower-right entries of $\SL(d)$. The case $d=3$ has
been studied in \cite{pauer1} and \cite{pauersu}. Let $B \subseteq G$ be the Borel subgroup
of upper triangular matrices, and let $G$ act linearly on $\C^d \times \C^d$ by acting
naturally on the first factor and with the contragredient action on the second
factor. We denote the coordinates of the first factor by $S_{11}, \ldots, S_{1d}$
and the coordinates of the second factor by $S_{21}, \ldots, S_{2d}$. Then the point
$((1,0,\ldots,0), (1,0,\ldots,0))$ has isotropy group $\SL(d-1)$, and its orbit is
$\V\big(\sum_{j=1}^d S_{1j}S_{2j} - 1\big)$. The homogeneous space
$G/H$ is spherical with $B$-invariant prime divisors $D_1 \coloneqq \V(S_{1d})$ and $D_2 \coloneqq \V(S_{21})$
as well as affine with factorial coordinate ring.
By Theorem~\ref{prop-trop}, we obtain $\Vm = \{v_1^* + v_2^* \le 0\}$.
Consider the embedding $Y$ corresponding to the fan containing
the one-dimensional cones having primitive lattice generators
$(p_1, q_1), \ldots, (p_n, q_n) \in \Nm$ with respect
to the basis $\{v_1, v_2\}$.
The following picture illustrates $\Nm_\Q$.
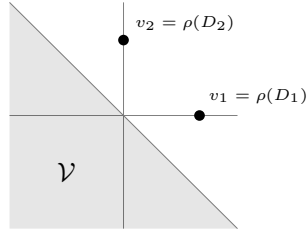
\begin{align*}
\begin{tikzpicture}
     \begin{scope}[help lines]
     \draw (-1.5,0) -- (1.5,0);
     \draw (0,-1.5) -- (0,1.5);
     \draw (-1.5, 1.5) -- (1.5, -1.5);
     \end{scope}
     \fill[opacity=0.1] (-1.5, 1.5) -- (1.5, -1.5) -- (-1.5, -1.5) -- cycle;
     \draw (1,-1pt) -- (1,1pt) node[anchor=south west]{\tiny{$v_1 = \orho(D_1)$}};
     \draw (-1pt,1) -- (1pt,1) node[anchor=south west]{\tiny{$v_2 = \orho(D_2)$}};
     \fill (1, 0) circle (2pt) ;
     \fill (0, 1) circle (2pt) ;
     \path (-0.75, -0.75) node {{$\mathcal{V}$}};
\end{tikzpicture}
\end{align*}
By Theorem~\ref{th-cox}, we obtain 
\begin{align*}
\Rm(Y) \cong \C[S_{1j}, S_{2j}, W_1, \ldots, W_n]_{j=1}^d \rleft/ \middle(\sum_{j=1}^d S_{1j}S_{2j} - W_1^{-p_1-q_1}\cdots W_n^{-p_n-q_n}\rright) \text{.}
\end{align*}
\end{example}

\begin{example}
Consider $\arb{G} \coloneqq \SL(2)$ and $\arb{H} \coloneqq T$,
the diagonal torus. Let $\arb{B} \subseteq \arb{G}$ be the
Borel subgroup of upper triangular matrices. Then, considering 
 the orbit of $([1:0], [0:1])$ in $\Pb^1 \times \Pb^1$, we obtain
$\arb{G}/\arb{H} \cong \Pb^1 \times \Pb^1 \setminus \operatorname{diag}\rleft(\Pb^1\rright)$,
and there are two $\arb{B}$-invariant prime divisors. From
\begin{align*}
\C\rleft[\arb{G}\rright] = \C[M_{11}, M_{12}, M_{21}, M_{22}]/\rleft(M_{11}M_{22} - M_{12}M_{21} - 1\rright)\text{,}
\end{align*}
we obtain $\arb{f}_1 = M_{21}$ and $\arb{f}_2 = M_{22}$. We can now construct $G = \arb{G} \times
\rleft(\C^*\rright)^{\arb{\Dm}}$
and $H$. It is not difficult to see that $G/H$ is the
orbit of the point $((1,0), (0,1), 1)$ in $\C^2 \times \C^2 \times \C^*$,
where $\arb{G}$ acts naturally on both factors $\C^2$ and trivially on $\C^*$ while
$\rleft(\C^*\rright)^{\arb{\Dm}}$
acts with the weights $-\eta_1$ and $-\eta_2$ respectively on the factors $\C^2$
as well as with $-\eta_1 -\eta_2$ on $\C^*$.
Denoting by $S_{11}$, $S_{12}$ and $S_{21}$, $S_{22}$ the coordinates of the factors $\C^2$
and by $T_1$ the coordinate of the factor $\C^*$, we obtain
$G/H = \V\rleft(S_{11}S_{22} - S_{12}S_{21} - T_1\rright)$.
By Theorem~\ref{prop-trop}, we obtain $\Vm = \{v_1^* + v_2^* \le w_1^*\}$. Using
Proposition~\ref{prop-composedmap} and Proposition~\ref{prop-ex-res},
we see that $\arb{\pi}_*(w_1)$ is a basis of $\arb{\Nm}_\Q$ and that
$\arb{\pi}_*(v_1) = \arb{\pi}_*(v_2) = -\arb{\pi}_*(w_1)$.
By Proposition~\ref{prop-valcone}, $\arb{\Vm} = \{\arb{\pi}_*(w_1)^* \ge 0\}$ follows.
The following picture illustrates $\arb{\Nm}_\Q$.
\begin{align*}
\begin{tikzpicture}
     \begin{scope}[help lines]
     \draw (-2,0) -- (2,0);
     \draw (0, -0.1) -- (0, 0.1);
     \end{scope}
     \fill[opacity=0.1] (2, -0.05) -- (0, -0.05) -- (0, 0.05) -- (2, 0.05) -- cycle;
     \draw (1,-1pt) -- (1,1pt) node[anchor=south]{\tiny{$\arb{\pi}_*(w_1)$}};
     \fill (-1, 0) circle (2pt) node[anchor=north]{\tiny{$\orho(\arb{D}_2)$}} node[anchor=south]{\tiny{$\orho(\arb{D}_1)$}} ;
     \path (1, 0) node[anchor=north] {\tiny{$\arb{\mathcal{V}}$}};
\end{tikzpicture}
\end{align*}
By Theorem~\ref{th-cox-gen}, we have
\begin{align*}
\Rm\rleft(\arb{G}/\arb{H}\rright) \cong \C[S_{11},S_{12},S_{21},S_{22}]/\rleft(S_{11}S_{22} - S_{12}S_{21}-1\rright)\text{.}
\end{align*}

The only nontrivial embedding of $\arb{G}/\arb{H}$ is the embedding
$\arb{Y} \coloneqq \Pb^1 \times \Pb^1$ given by the cone $\Q_{\ge 0}\arb{\pi}_*(w_1)$.
In this case, Theorem~\ref{th-cox-gen} yields
\begin{align*}
\Rm\rleft(\arb{Y}\rright) \cong \C[S_{11},S_{12},S_{22},S_{21},W_1]/\rleft(S_{11}S_{22} - S_{12}S_{21}-W_1\rright)\text{,}
\end{align*}
which is isomorphic to the polynomial ring in four variables.
\end{example}

\begin{example}[{cf.~\cite[Example~4.2]{cfact}}]
Consider $\arb{G} \coloneqq \SL(2)$ and $\arb{H} \coloneqq N_{\arb{G}}\rleft(T\rright)$
where $T$ is the diagonal torus.
Let $\arb{B} \subseteq \arb{G}$ be the Borel subgroup of upper triangular matrices.
The group $\arb{H}$ consists of two connected components, the identity
component
$\arb{H}^\circ = \big\{\big(\begin{smallmatrix}\lambda & 0 \\ 0 & \lambda^{-1}\end{smallmatrix}\big); \lambda \in \C^*\big\}$ and a second component
$\arb{H}^\dagger = \big\{\big(\begin{smallmatrix} 0 & \lambda & \\ -\lambda^{-1} & 0\end{smallmatrix}\big); \lambda \in \C^*\big\}$. There exists a character $\arb{\chi} \in \X\rleft(\arb{H}\rright)$
with $\arb{\chi}|_{\arb{H}^\circ} = 1$ and $\arb{\chi}|_{\arb{H}^\dagger} = -1$.
Consider the complex vector space $\Sym(2\times 2) = 
\big\{\big(\begin{smallmatrix}s_{12} & s_{13} \\ s_{13} & s_{11}\end{smallmatrix}\big); s_{ij} \in \C^*\big\}$
of symmetrical 2 by 2 matrices. Let $\arb{G} \times \C^*$
act on $\Sym(2\times 2) \times \C^*$ via $(g, t)\cdot (x, y) \coloneqq (t^{-1}gxg^T, t^{-2}y)$. The isotropy
group of the point $\big(\big(\begin{smallmatrix}0 & 1 \\ 1 & 0\end{smallmatrix}\big), 1\big)$ is $H' \coloneqq \rleft\{(h, \arb{\chi}(h)), h \in \arb{H}\rright\}$ and its orbit is the
closed subset $\V(S_{11}S_{12} - S_{13}^2 - T_1)$. We observe that $\rleft(\arb{G} \times \C^*\rright)/H'$ is
a spherical homogeneous space with one $\arb{B} \times \C^*$-invariant
prime divisor $D_1 \coloneqq \V(S_{11})$ and that it
coincides with $G/H$ as defined in Section~\ref{three}. By Theorem~\ref{prop-trop}, we
obtain $\Vm = \{2v_1^* \le w_1^*\}$. Using
Proposition~\ref{prop-composedmap} and Proposition~\ref{prop-ex-res},
we see that $\arb{\pi}_*(w_1)$ is a basis of $\arb{\Nm}_\Q$ and that
$\arb{\pi}_*(v_1) = -2\arb{\pi}_*(w_1)$.
By Proposition~\ref{prop-valcone}, $\arb{\Vm} = \{\arb{\pi}_*(w_1)^* \ge 0\}$ follows.
The following picture illustrates $\arb{\Nm}_\Q$.
\begin{align*}
\begin{tikzpicture}
     \begin{scope}[help lines]
     \draw (-2.5,0) -- (2.5,0);
     \draw (0, -0.1) -- (0, 0.1);
     \end{scope}
     \fill[opacity=0.1] (2.5, -1pt) -- (0, -1pt) -- (0, 1pt) -- (2.5, 1pt) -- cycle;
     \draw (2,-1pt) -- (2,1pt) ;
     \draw (-1,-1pt) -- (-1,1pt) ;
     \draw (1,-1pt) -- (1,1pt) node[anchor=south]{\tiny{$\arb{\pi}_*(w_1)$}};
     \fill (-2, 0) circle (2pt) node[anchor=south]{\tiny{$\orho(\arb{D_1})$}} ;
     \path (1.25, 0) node[anchor=north] {\tiny{$\arb{\mathcal{V}}$}};
\end{tikzpicture}
\end{align*}
Consider the trivial embedding $\arb{Y} \coloneqq \arb{G}/\arb{H}$.
From a geometric viewpoint, this is the complement
of a conic in $\Pb^2$.
We have $\Cl\rleft(\arb{Y}\rright) \cong \Z/2\Z$ with generator 
$\rleft[\arb{D}_1\rright] \in \Cl\rleft(\arb{Y}\rright)$.
By Theorem~\ref{th-cox-gen}, we obtain
\begin{align*}
\Rm\rleft(\arb{Y}\rright) \cong \C[S_{11}, S_{12}, S_{13}]/(S_{11}S_{12} - S_{13}^2 - 1)\text{,}
\end{align*}
which is a non-factorial ring.
\end{example}

\section{Comparison with the approach of Brion}
\label{five}

When $Y$ is a spherical $G$-variety satisfying $\glob{Y} = \C$, Brion defines the
\em equivariant Cox ring \em of $Y$, which is graded by 
the group $\Cl_{G}(Y)$ of isomorphism classes of $G$-linearized
divisorial sheaves on $Y$, as
\begin{align*}
\Rm_{G}(Y) \coloneqq \bigoplus_{[\Fm] \in \Cl_{G}(Y)} \Gamma(Y, \Fm)\text{,}
\end{align*}
where again some accuracy is required in order to define a multiplication law (cf.~\cite[4.2]{brcox}).
We have $\Rm_G(Y) \cong \Rm(Y) \otimes_\C \C[C]$, where the characters of the torus
$C$ correspond to the various choices of linearizations.

In order to obtain a description of $\Rm_G(Y)$, Brion first computes the Cox ring
of the associated wonderful variety $Y'$.
According to \cite[6.1]{Luna:typea}, the wonderful variety $Y'$ is obtained as follows.
The equivariant automorphism group of $G/H$ can be identified with $N_G(H)/H$, hence
$N_G(H)$ acts on the set $\Dm$ of $B$-invariant prime divisors in $G/H$.
The subgroup of $N_G(H)$ which stabilizes $\Dm$ is called the \em spherical closure \em of $H$. 
It can be written as $H' \times C$ where $H' \subseteq G' \coloneqq G^{ss}$.
Then $G'/H'$ is a spherical homogeneous space admitting
a wonderful completion $G'/H' \hookrightarrow Y'$.

We consider the Borel subgroup $B' \coloneqq B \cap G'$ as well as the maximal torus $T' \coloneqq T \cap G'$,
define $\X \coloneqq \X(B') = \X(T')$, and denote by $\X^+ \subseteq \X$ the subset of dominant weights.
There is a natural inclusion $\epsilon: \X \hookrightarrow \C[T']^*$, $\chi \mapsto \epsilon^\chi$.
The set of $B'$-invariant prime divisors in $G'/H'$ can be identified with $\Dm = \{D_1, \ldots, D_r\}$.
We denote by $\omega_i$ the left $B'$-weight of the equation $f'_i$ of the
pullback of $D_i$ under $G' \to G'/H'$. There is a natural bijection
between the set $\{\gamma_1, \ldots, \gamma_s\}$ of spherical roots of $G'/H'$
and the set  $\{Y'_1, \ldots, Y'_s\}$ of $G'$-invariant
prime divisors in $Y'$.

The Cox ring $\Rm(Y')$ is realized as a subring of
\begin{align*}
\C[G' \times T'] = \C[G'] \otimes_{\C} \C[T'] = \bigoplus_{\eta \in \X} \C[G']\otimes_\C \C\epsilon^\eta\text{.}
\end{align*}
We will require the following notation: when $V$ is a $G'$-module and $\chi \in \X^+$, we denote by $V_{(\chi)}$
the corresponding isotypic component of $V$, i.e.~the sum of all
irreducible $G'$-submodules of highest weight $\chi$.
We denote by $s_l \in \Gamma(Y', \Om_{Y'}(Y'_l))$ and $s_{D_i} \in \Gamma(Y', \Om_{Y'}(D_i))$
the respective canonical sections.
Finally, we denote by $K'$ the intersection of the kernels of all characters of $H'$.
Note that $G'/K'$ is always quasiaffine, but
may fail to have trivial divisor class group when
$H'$ is not connected.

\begin{theorem}[cf.~{\cite[Theorem~3.2.3]{brcox}}]
\label{thm:brionw}
We define the following partial order on $\X$: we write $\chi \preceq \eta$ if $\eta - \chi$
is a linear combination of spherical roots of $G'/H'$ with nonnegative coefficients.
Then we have
\begin{align*}
\Rm(Y') \cong \bigoplus_{\substack{(\chi, \eta) \in \X^+ \times \X\\ \chi \preceq \eta}}
\glob{G'/K'}_{(\chi)} \otimes_\C \C\epsilon^\eta\text{.}
\end{align*}
This isomorphism sends $s_l \mapsto 1 \otimes \epsilon^{\gamma_l}$ and $s_{D_i} \mapsto f'_i \otimes \epsilon^{\omega_i}$.
Furthermore, $\Rm(Y')$ is the Rees algebra associated to
the ascending filtration
\begin{align*}
\mathscr{F}^{\eta}(\glob{G'/K'}) \coloneqq \bigoplus_{\substack{\chi\in\X \\ \chi \preceq \eta}} \glob{G'/K'}_{(\chi)}
\end{align*}
of $\glob{G'/K'}$ indexed by the partially ordered set $\X$.
\end{theorem}

\begin{example}
\label{ex:w}
Let $G' \coloneqq \SL(d)$ for $d \ge 3$ act on $Y' \coloneqq \Pb^d \times \Pb^d$
by acting naturally on the first factor and with the
contragredient action on the second factor.
Then $Y'$ is a wonderful variety and the point $((1:0:\ldots:0), (1:0:\ldots:0)) \in Y'$ with
isotropy group $H' \coloneqq \GL(d-1)$ lies in the open $G'$-orbit.
We have
\begin{align*}
\Rm(Y') = \C[X_1, \ldots, X_d, Y_1, \ldots, Y_d]\text{.}
\end{align*}
There is exactly one $G$-invariant prime divisor
$D \coloneqq \V\rleft(X_1Y_1 + \ldots + X_dY_d\rright)$.
The intersection $K'$ of the kernels of all characters of $H'$ is $\SL(d-1)$ as in Example~\ref{ex:comp}.
In particular, we have
\begin{align*}
R \coloneqq \glob{G'/K'} = \C[X_1, \ldots, X_d, Y_1, \ldots, Y_d] \rleft/ \middle(X_1Y_1 + \ldots + X_dY_d - 1\rright) \text{.}
\end{align*}
With the notation from Example~\ref{ex:comp}, $\gamma_1 \coloneqq v_1^*+v_2^* \in \X(B')$ is 
the unique spherical root of $G'/H'$. Theorem~\ref{thm:brionw} yields
an isomorphism
\begin{align*}
\Rm(Y') \cong \bigoplus_{i,j,k=0}^\infty R_{(iv_1^*+jv_2^*)} \otimes_\C \C\epsilon^{iv_1^*+jv_2^*+k\gamma_1}
\end{align*}
sending $X_i \mapsto X_i \otimes \epsilon^{v_1^*}$ and $Y_i \mapsto Y_i \otimes \epsilon^{v_2^*}$. In particular, it sends
\begin{align*}
X_1Y_1 + \ldots + X_dY_d \mapsto 1 \otimes \epsilon^{\gamma_1}\text{.}
\end{align*}
\end{example}

Now assume that the fan associated to the spherical embedding $G/H \hookrightarrow Y$ 
contains exactly the trivial cone $0$ and the one-dimensional cones with primitive lattice generators $u_1, \ldots, u_n \in \Nm$.
The natural
map $G/H \to G'/H'$ extends to a $G$-equivariant morphism
$\phi: Y \to Y'$,
which induces an equivariant homomorphism of graded algebras
\begin{align*}
\phi^*: \Rm(Y') = \Rm_{G'}(Y') \to \Rm_{G}(Y)\text{.}
\end{align*}
The invariant subrings
\begin{align*}
\Rm_{G}(Y)^{G'} &\cong \C[C][W_1, \ldots, W_n]\\
\Rm(Y')^{G'} &\cong \C[Z_1, \ldots, Z_s]
\end{align*}
are polynomial rings over $\C[C]$ and $\C$
where $W_1, \ldots, W_n$ and $Z_1, \ldots, Z_s$ are identified with the canonical sections
corresponding to the $G$-invariant prime divisors in $Y$ and $Y'$ respectively.

\begin{theorem}[{cf.~\cite[Theorem~4.3.2]{brcox}}]
The restricted map
\begin{align*}
\phi^*|_{\Rm(Y')^{G'}}: \Rm(Y')^{G'} \to \Rm_{G}(Y)^{G'}
\end{align*}
sends $Z_i \mapsto W_1^{-\langle u_1, \gamma_i \rangle} \cdots W_n^{-\langle u_n, \gamma_i \rangle}$, and the map
\begin{align*}
\Rm(Y') \otimes_{\Rm(Y')^{G'}} \Rm_{G}(Y)^{G'} \to \Rm_{G}(Y) 
\end{align*}
sending $s' \otimes s \mapsto \phi^*(s')s$ is an isomorphism.
\end{theorem}

\begin{remark}
As $\Rm(Y) \cong \Rm_G(Y)^C$ (cf.~\cite[proof of Theorem~4.3.2]{brcox}), it
immediately follows that
\begin{align*}
\Rm(Y) \cong \Rm(Y') \otimes_{\C[Z_1, \ldots, Z_s]} \C[W_1, \ldots, W_n]\text{.}
\end{align*}
Furthermore, if $Y$ is horospherical and $P \coloneqq N_G(H)$, we have $Y' \cong G/P$
and $s = 0$, hence $\Rm(Y) \cong \Rm(G/P)[W_1, \ldots, W_n]$ (cf.~Theorem~\ref{th-cox-hor}).
\end{remark}

\begin{example}
Consider $G \coloneqq \SL(d)$ for $d \ge 3$ and $H \coloneqq \SL(d-1)$
as in Example~\ref{ex:comp}. Then $G' = G$, and the spherical closure of $H$
is $H' \coloneqq \GL(d-1)$ as in Example~\ref{ex:w}.
Consider the embedding $G/H \hookrightarrow Y$ corresponding to the fan containing
the one-dimensional cones having primitive lattice generators
$(p_1, q_1), \ldots, (p_n, q_n) \in \Nm$ with respect
to the basis $\{v_1, v_2\}$ as in Example~\ref{ex:comp}. There is
exactly one spherical root $\gamma_1 = v_1^* + v_2^*$.
As $G$ is semisimple, it follows that $\Rm_{G}(Y) = \Rm(Y)$.
We have
\begin{align*}
\Rm(Y') &= \C[X_1, \ldots, X_d, Y_1, \ldots, Y_d]\\
\Rm(Y')^{G} &= \C[Z_1] \\
\Rm(Y)^{G} &= \C[W_1, \ldots, W_n]\text{,}
\end{align*}
the inclusion
$\C[Z_1] \hookrightarrow \C[X_1, \ldots, X_d, Y_1, \ldots, Y_d]$ is given
by $Z_1 \mapsto X_1Y_1 + \ldots + X_dY_d$,
and $\phi^*: \C[Z_1] \to \C[W_1, \ldots, W_s]$ sends
$Z_1 \mapsto W_1^{-p_1-q_1}\cdots W_n^{-p_n-q_n}$.
From
\begin{align*}
\Rm(Y) \cong \C[X_1, \ldots, X_d, Y_1, \ldots, Y_d] \otimes_{\C[Z_1]} \C[W_1, \ldots, W_n]\text{,}
\end{align*}
we obtain
\begin{align*}
\Rm(Y) \cong \C[X_j,Y_j, W_1, \ldots, W_n]_{j=1}^d \rleft/ \middle(\sum_{j=1}^d X_{j}Y_{j} - W_1^{-p_1-q_1}\cdots W_n^{-p_n-q_n}\rright) \text{,}
\end{align*}
which is in agreement with Example~\ref{ex:comp}.
\end{example}

\section*{Acknowledgments}
I am highly grateful to my teacher, Victor Batyrev, for his guidance and support.
I would also like to thank J\"{u}rgen Hausen and Johannes Hofscheier for
some useful discussions. Finally, I thank the referee for carefully
reading the manuscript as well as for valuable comments and suggestions.

\bibliographystyle{amsalpha}
\bibliography{crse}

\end{document}